\documentclass[11pt]{amsart}

\usepackage[utf8]{inputenc}

\title{Separating cyclic subgroups in graph products of groups}
\author{Federico Berlai}
\address[Federico Berlai]{Department of Mathematics, UPV/EHU, Barrio Sarriena s/n, 48940, Leioa, Spain}
\email[Federico Berlai]{federico.berlai@ehu.eus}

\author{Michal Ferov}
\address[Michal Ferov]{School of Mathematical and Physical Sciences, University of Technology Sydney, PO Box 123
Broadway, NSW 2007, Australia}
\email[Michal Ferov]{michal.ferov@gmail.com}

\usepackage{amsthm}
\usepackage{amsmath, amssymb}
\usepackage{url}
\usepackage{color}
\usepackage{verbatim}

\usepackage{todonotes}

\usepackage{amsthm}
\newtheorem{theorem}{Theorem}[section]
\newtheorem{proposition}[theorem]{Proposition}
\newtheorem{lemma}[theorem]{Lemma}
\newtheorem{corollary}[theorem]{Corollary}

\theoremstyle{definition}
\newtheorem{definition}[theorem]{Definition}
\newtheorem{example}[theorem]{Example}
\newtheorem{question}[theorem]{Question}
\newtheorem{remark}[theorem]{Remark}

\newtheorem{thmx}{Theorem}
\newtheorem{corx}{Corollary}

\newtheorem*{thmB}{Theorem \ref{proposition:GP_p-CSS}}
\newtheorem*{thmBB}{Theorem \ref{general:GP_CSS}}

\DeclareMathOperator{\Inn}{Inn}
\DeclareMathOperator{\BS}{BS}
\DeclareMathOperator{\C}{\mathcal{C}}
\DeclareMathOperator{\NC}{\mathcal{N_C}}
\DeclareMathOperator{\Np}{\mathcal{N}_{\mathnormal{p}}}

\DeclareMathOperator{\PT}{\mathcal{PT}}
\DeclareMathOperator{\proC}{pro-\mathcal{C}}
\DeclareMathOperator{\prop}{pro-{\mathnormal{p}}}

\DeclareMathOperator{\Rad}{Rad}
\DeclareMathOperator{\plog}{plog}
\DeclareMathOperator{\supp}{supp}
\DeclareMathOperator{\ord}{ord}
\DeclareMathOperator{\esupp}{esupp}
\DeclareMathOperator{\link}{link}
\let\star\relax
\DeclareMathOperator{\star}{star}
\DeclareMathOperator{\id}{id}

\DeclareMathOperator{\LL}{LL}
\DeclareMathOperator{\FL}{FL}

\DeclareMathOperator{\E}{E}

\DeclareMathOperator{\U}{\mathfrak{U}}
\DeclareMathOperator{\Ups}{\mathfrak{U}_{ps}}

\begin{document}

\maketitle

\begin{abstract}
    We prove that the property of being cyclic subgroup separable, that is having all cyclic subgroups closed in the profinite topology, is preserved under forming graph products. 
    
    Furthermore, we develop the tools to study the analogous question in the pro-$p$ case. For a wide class of groups we show that the relevant cyclic subgroups - which are called $p$-isolated - are closed in the pro-$p$ topology of the graph product. In particular, we show that every $p$-isolated cyclic subgroup of a right-angled Artin group is closed in the pro-$p$ topology, and we fully characterise such subgroups.
\end{abstract}

\tableofcontents

\section{Introduction}
A very successful way to understand countable discrete groups certainly is by studying them through their finite quotients, and groups where this approach works to its full extent are called residually finite. A group $G$ is \emph{residually finite} if for any two distinct elements $g_1, g_2 \in G$ there is a finite group $Q$ and a surjective homomorphism $\pi \colon G \to Q$ such that $\pi(g_1)$ is distinct from $\pi(g_2)$ in $Q$. 

In general, properties of this type are called \emph{separability properties}: a subset $S \subseteq G$ is said to be \emph{separable} in $G$ if for every $g \in G \setminus S$ there exists a finite quotient of $G$ such that the image of $g$ under the canonical projection does not belong to the image of $S$. 
This is equivalent to $S$ being closed in the profinite topology of $G$.
Therefore, a group is residually finite if and only if the subset $\{e_G\}$ is separable.

One can then define a separability property by specifying which subsets are required to be separable: \emph{conjugacy separable} groups have separable conjugacy classes of elements, LERF groups (also called locally extended residually finite) have separable finitely generated subgroups, and \emph{cyclic subgroup separable} groups (also denoted CSS groups or $\pi_c$) are those where all cyclic subgroups are separable.

In a way, the notion of separability gives an algebraic analogue to decision problems in finitely presented groups: if the subset $S \subseteq G$ is given in suitably nice way  (meaning that $S$ is recursively enumerable and one can always effectively construct the image of $S$ under the canonical projection onto a finite quotient of $G$) and it is separable in $G$, one can then decide whether a word in the generators  of $G$ represents an element belonging to $S$ simply by checking finite quotients. Indeed, it was proved by Mal'cev \cite{malcev} that finitely presented residually finite groups have solvable word problem, and Mostowski \cite{mostowski} showed that finitely presented conjugacy separable groups have solvable conjugacy problem. In a similar fashion, finitely presented LERF groups have solvable generalised word problem, meaning that the membership problem is uniformly solvable for every finitely generated subgroup. In general, algorithms that involve enumerating finite quotients of an algebraic structure are called algorithms of Mal'cev-Mostowski type. 

\smallskip
In this paper we study the already mentioned CSS groups, the groups in which all cyclic subgroups are separable. Naturally, one may describe finitely presented CSS groups as the groups in which the \emph{power problem}, i.e. given words $u$ and $v$ the problem of deciding whether $u$ represents an element that is a power of the element represented by $v$, can be solved by an algorithm of Mal'cev-Mostowski type.

The main focus of this note is to study how does the property of being cyclic subgroup separable behave with respect to certain group-theoretic constructions. In particular, we study separability of cyclic subgroups in \emph{graph products of groups}, a natural generalisation of free and direct products in the category of groups, first introduced by Green \cite{green}. Let $\Gamma = (V\Gamma, E\Gamma)$ be a simplicial graph, i.e. $V\Gamma$ is a set and $E\Gamma \subseteq \binom{V\Gamma}{2}$, and let $\mathcal{G} = \{G_v \mid v \in V\Gamma\}$ be a collection of groups. The graph product $\Gamma\mathcal{G}$ is the quotient of the free product $\ast_{v \in V\Gamma}G_v$ modulo all the relations of the form
\begin{displaymath}
    g_u g_v = g_v g_u \mbox{ for all } g_u \in G_u, g_v \in G_v, \{u,v\}\in E\Gamma.
\end{displaymath}
The groups $G_v \in \mathcal{G}$ are referred to as \emph{vertex groups}. Well-known examples of graph products are \emph{right-angled Artin groups} (raags) and \emph{right-angled Coxeter groups}, where all vertex groups are respectively infinite cyclic, or cyclic groups of order two. 

When $\Gamma$ is totally disconnected, that is the edge set is empty, the resulting graph product is the free product of the vertex groups, and if $\Gamma$ is complete, that is any two vertices are joined by an edge, the resulting graph product is the direct sum of the vertex groups.

Green proved that the class of residually finite groups is closed with respect to forming graph products \cite[Corollary 5.4]{green}. On the other hand, Minasyan proved that raags are conjugacy separable \cite[Theorem 1.1]{ashot_raags}, and this result was later generalised by the second-named author to show that the class of conjugacy separable groups is closed under forming graph products \cite[Theorem 1.1]{mf}.

Closure properties of the class of CSS groups have been studied previously. 
Ribes, Segal and Zaleskii considered the class of quasi-potent CSS groups, and proved that it is closed under taking amalgamations along cyclic groups \cite[Proposition 4.2]{cs_amalgams_zalesskii}. Analogous conclusions were obtained previously by Evans \cite[Lemma 2.4]{evans74}.
Later, Burillo and Martino \cite{quasipotency} showed that the same class is closed with respect to special HNN-extensions along a separable subgroup, amalgamations along a virtually-cyclic subgroup, and direct products.
Allenby and Gregorac \cite[Theorem 2']{allenby} (compare also Bobrovskii and Sokolov \cite{bobr}) showed that the class of CSS groups is closed under amalgamation along a common retract. More recently, Chagas and Zaleskii \cite[Theorem 1.4]{hcs_amalgams_zalesskii} showed that right-angled Artin groups and their virtual retracts have many separability properties, one of which implies being CSS.

\smallskip
The first result of this work is to prove that the class of cyclic subgroup separable groups is closed under forming graph products. This generalises a result of Green, who showed that right-angled Artin groups are CSS \cite[Theorem 2.16]{green}.
\begin{thmx}
\label{proposition:GP_CSS}
\emph{The class of CSS groups is closed under forming graph products}.
\end{thmx}
In particular, Theorem \ref{proposition:GP_CSS} follows from the more general:
\begin{thmx}
\label{general:GP_CSS}\emph{
Let $\mathcal{P}$ be a property of groups that is preserved under finite direct products and under amalgamation over retracts. Then $\mathcal{P}$ is preserved under finite graph products.}
\end{thmx}
We refer to Section \ref{section:graph_products} and Section \ref{section:amalgams} for the formal definitions of graph product and amalgamation over retract, respectively.

\smallskip
Let us stress that Theorem \ref{proposition:GP_CSS} cannot be strengthened to (nor can Theorem \ref{general:GP_CSS} be applied to) LERF groups, that is to groups where all finitely generated subgroups are closed in the profinite topology. Indeed, the class of LERF groups is not closed under forming direct products: free groups are LERF \cite{hall}, but the group $F_2 \times F_2$ contains finitely generated subgroups with unsolvable membership problem \cite{mihailova}, hence it cannot be LERF.

\medskip
The notion of separability can be generalised in a natural way by considering only certain kind of quotients: let $\mathcal{C}$ be a class of groups, we then say that a subset $S \subseteq G$ is $\mathcal{C}$-separable in $G$ if for every $g \in G \setminus S$ there is a group $Q \in \mathcal{C}$ and an epimorphism $\pi \colon G \to Q$ such that $\pi(g)$ does not belong to $\pi(S)$. 

Toinet \cite{Toinet} extended Minasyan's result about conjugacy separability in raags, showing that conjugacy classes in raags are separable with respect to the class of all finite $p$-groups. Toinet's result was later generalised by the second-named author \cite{mf}, who proved that the class of $\C$-conjugacy separable groups is closed under forming graph products whenever $\C$ is an extension closed pseudovariety of finite groups. 

The original result of Green on residual finiteness being preserved by graph products was extended by the authors \cite[Theorem A]{bf}, where it was shown that the class of residually $\C$ groups is closed with respect to forming graph products for many classes $\mathcal{C}$ of groups, including solvable or amenable groups.

\medskip
After proving Theorem \ref{proposition:GP_CSS}, we shift our attention to cyclic subgroup separability in the context of pro-$p$ topologies.
We consider the class of all finite $p$-groups, where $p$ is some prime number, and we say that a subset $S \subseteq G$ is $p$-separable in $G$ if it is closed in the pro-$p$ topology of $G$. 
This case is subtler than the one of profinite topology: all subgroups of $\mathbb{Z}=\langle a\rangle$ are closed in the profinite topology, but $\langle a^2\rangle$ is not $p$-separable in $\mathbb{Z}$ for any odd prime $p$.

Following \cite{bardakov, bobr} (compare also Definition \ref{def:p.isolation}) we say that the subgroup $H \leq G$ is $p$-\emph{isolated} in $G$ if for any $g \in G$ and any prime $q$ distinct from $p$ the following implication holds: whenever $g^q \in H$, then already $g\in H$.
We say that $G$ is \emph{cyclic subgroup} $p$-\emph{separable} ($p$-CSS for short) if all its $p$-isolated cyclic subgroups are $p$-separable. 

To be able to easily identify $p$-isolated subgroups, we develop the notion of primitive stability in the context of groups with unique roots (see respectively Definition \ref{definition:plogs} and Definition \ref{def:unique_roots}). Roughly speaking, $G$ is a primitively stable group with unique roots if expressing elements of $G$ as powers of ``smaller'' elements behaves in a predictable manner, similar to infinite cyclic groups. We refer to Section \ref{section:primitive} for the precise definition of primitive stability and its connection to $p$-isolation.

To ease the notation, let $\Ups$ denote the class of primitively stable groups with unique roots. After showing that $\Ups$ is closed under taking graph products (compare Theorem \ref{lemma:primitive_urp_gp}), we prove an analogue of Theorem \ref{proposition:GP_CSS} for pro-$p$ topologies.
\begin{thmx}
    \label{proposition:GP_p-CSS}
\emph{    For every prime number $p$, the class of $p$-CSS groups in $\Ups$ is closed under forming graph products}.
\end{thmx}

As the infinite cyclic group belongs to $\Ups$ and it is $p$-CSS for all prime numbers $p$, we obtain:
\begin{corx}
    \label{corollary:raags_pcss}
\emph{For every prime number $p$, right-angled Artin groups are \mbox{$p$-CSS}. In particular, every maximal cyclic subgroup of a right-angled Artin group is $p$-separable in the group for every prime $p$}.
\end{corx}

Groups in the class $\Ups$ admit a more algebraic characterisation. Let $g\in G$, and define the \emph{radical} of $g$ in $G$ as the subset
\[\Rad_G(g)=\bigl\{r\in G\mid r^a \in \langle g\rangle \text{ for some } a \in\mathbb{Z}\setminus\{0\}\bigr\}.\]
The radical of an element $g\in G$ is not in general a subgroup, but this is the case for groups in $\Ups$. Moreover, we characterise the groups in $\Ups$ as the torsion-free groups where $\Rad_G(g)$ is a cyclic subgroup for every non-trivial element. 
\begin{thmx}\label{thm:C}
\emph{Let $G$ be a torsion-free group. The following are equivalent:
\begin{enumerate}
    \item $G\in\Ups$;
    \item for all non-trivial $g\in G$ the subset $\Rad_G(g)$ is an infinite cyclic subgroup of $G$.
\end{enumerate}
}
\end{thmx}

Exploiting Theorem \ref{thm:C}, we prove that the class $\Ups$ is quite wide: torsion-free hyperbolic groups, (residually) finitely generated torsion-free nilpotent groups, and torsion-free groups hyperbolic relative to groups in $\Ups$, all belong to $\Ups$. In particular, limit groups belong to $\Ups$, being toral relatively hyperbolic. As already stated, the class $\Ups$ is closed under graph products.

\medskip
The paper is organised as follows. In Section \ref{section:proC_topologies} we recall the notion of profinite and $\proC$ topologies on groups, and review the classical results that allow us to use topological methods when working with separability properties; readers familiar with $\proC$ topologies might feel free to skip this section. In Section \ref{section:graph_products} we recall the notation for graph products of groups and review known properties, such as Normal Form Theorem and the definition of cyclically reduced elements and full subgroups. 

In Section \ref{section:amalgams} we introduce the notion of amalgams over a common retract, develop the standard combinatorial framework for this type of free-construction and show how it relates to graph products in a natural way: we show that, as soon as it is not a direct product of its vertex groups, a graph product splits as an amalgam of its proper full subgroups over a common retract. Using the previously developed framework, in Section \ref{section:profinite_case} we prove Theorem \ref{proposition:GP_CSS}. 

In Section \ref{section:unique_roots} we recall the notion of Unique Root property and, using the framework for amalgams over retracts, we show that the the class of groups with Unique Root property - which we denote by $\U$ - is closed under forming graph products. In Section \ref{section:primitive} we introduce the notion of primitive logarithms, primitive roots and primitive stability. In the class $\Ups$ of primitively stable groups with unique roots we can give a characterisation of $p$-isolated cyclic subgroups in terms of the primitive logarithm of the generator of the cyclic subgroup. 

In Section \ref{section:primitive_stabilility_closure} we study the closure properties of the class $\Ups$. In particular, we show that it is closed with respect to taking subgroups, direct products, amalgams over retracts and, consequently, graph products. In Section \ref{section:examples_of_Ups} we prove Theorem \ref{thm:C}, and we show that the class $\Ups$ contains all finitely generated torsion-free nilpotent groups, all torsion-free hyperbolic groups, all limit groups, and certain relatively hyperbolic groups. A proof of Theorem \ref{proposition:GP_p-CSS} is given in Section \ref{section:prop_case}. Finally, in Section \ref{section:p-isolated_elements} we give a characterisation of $p$-isolated cyclic subgroups in graph products of groups.

\section{Pro-$\C$ topologies on groups}
\label{section:proC_topologies}
This section, after a brief paragraph on notation, contains basic facts about $\proC$ topologies on groups; we are including it to make the paper self-contained and experts can feel free to skip it. Proofs of all the statements can be found in the classic book by Ribes and Zalesskii \cite{rz} or in the second named author's thesis \cite{mf_thesis}.

If $G$ is a group, then $e_G$, or $e$ when the group $G$ is clear from the context, denotes the identity element in $G$. For elements $g,h \in G$ we will use $g^h$ to denote $h g h^{-1}$, the $h$-conjugate of $g$. Similarly, for a subgroup $H \leq G$ we will use $g^H$ to denote $\{hgh^{-1} \mid h \in H\}$.
In this note the natural numbers $\mathbb{N}$ include zero.

\smallskip
Let $\mathcal{C}$ be a class of groups and let $G$ be a group. We say that a normal subgroup $N \unlhd G$ is a \emph{co-$\mathcal{C}$ subgroup} of $G$ if $G/N \in \mathcal{C}$, and we denote by 
$\NC(G)$ the set of co-$\mathcal{C}$ subgroups of $G$. 

Consider the following closure properties for a class of groups $\mathcal{C}$:
	\begin{itemize}
	    \item[(c0)] $\mathcal{C}$ is closed under taking finite subdirect 
	    products,
	    \item[(c1)]	$\mathcal{C}$ is closed under taking subgroups,
	    \item[(c2)]	$\mathcal{C}$ is closed under taking finite direct products.
	\end{itemize}
Note that 
\[(c0) \Rightarrow  (c2)\qquad\text{and}\qquad (c1)+ (c2) \Rightarrow (c0).\]
If the class $\mathcal{C}$ satisfies (c0) then, for every group $G$, the set $\NC(G)$ is closed under finite intersections, that is to say, if $N_1, N_2\in\NC(G)$ then also $N_1\cap N_2\in\NC(G)$. This
implies that $\NC(G)$ is a base at $e_G$ for a topology on $G$.

Hence the group $G$ can be equipped with a group topology, where the base of open sets is given by 
\[\{gN \mid g\in G,  N \in \NC(G)\}.\]
This topology, denoted by pro-$\mathcal{C}(G)$, is called the \emph{pro-$\mathcal{C}$ topology} on $G$.

If the class $\mathcal{C}$ satisfies (c1) and (c2), or equivalently, (c0) and (c1), then one can easily see that equipping a group $G$ with its pro-$\mathcal{C}$ topology is a faithful functor from the category of groups to the category of topological groups. 
\begin{lemma}
	\label{lemma:continuous}
	Let $\C$ be a class of groups satisfying (c1) and (c2). Given groups $G$ and $H$, every homomorphism $\varphi\colon G\to H$ is a continuous map with respect to the corresponding pro-$\mathcal{C}$ topologies. Furthermore, if $\varphi$ is an isomorphisms, then it is a homeomorphism.
\end{lemma}

A set $X \subseteq G$ is $\mathcal{C}$-\emph{closed} in $G$ if $X$ is closed in pro-$\mathcal{C}(G)$: for every $g \notin X$ there exists $N \in \NC(G)$ such that the open set $gN$ does not intersect $X$, that is, 
$gN \cap X = \emptyset$.
This is equivalent to $gN \cap XN=\emptyset$, and hence $\varphi(g) \notin \varphi(X)$ in $G/N$, where $\varphi \colon G \twoheadrightarrow G/N$ is the canonical projection onto the quotient $G/N$. 
Accordingly, a set is $\mathcal{C}$-\emph{open} in $G$ if it is open in pro-$\mathcal{C}(G)$.

In this paper we will only consider the class $\C$ of all finite groups or of all finite $p$-groups, and therefore we will assume that $\C$ is closed under subgroups, finite direct products, quotients and extensions. A class of groups satisfying these properties is also called an extension-closed pseudovariety of finite groups.

\smallskip
In the following lemma we collect known facts about open and closed subgroups, in particular \cite[Theorem 3.1, Theorem 3.3]{MHall50}.

\begin{lemma}
\label{lemma:C-open}
\label{lemma:C-closed}

Let $G$ be a group and let $H \leq G$. Then
\begin{itemize}
    \item[(i)] $H$ is $\mathcal{C}$-open in $G$ if and only if there is $N \in \NC(G)$ such that $N \leq H$; moreover, every $\mathcal{C}$-open subgroup is of finite index in $G$ and it is $\mathcal{C}$-closed in $G$;
    \item[(ii)] $H$ is $\C$-closed in $G$ if and only if $H$ is an intersection of open subgroups.
\end{itemize}
\end{lemma}

\begin{remark}
    \label{remark:useful}
    Let $G_1, G_2$ be groups and suppose that $H_1 \leq G_1$, $H_2 \leq G_2$ are $\C$-closed in $G_1$ and $G_2$ respectively. Then $H_1 \times H_2$ is $\C$-closed in $G_1 \times G_2$.
\end{remark}

\subsection{Restrictions of pro-$\mathcal{C}$ topologies}
Let $G$ be a group and let $H \leq G$. We say that that $\proC(H)$ is a \emph{restriction} of $\proC(G)$ if a subset $X \subseteq H$ is $\C$-closed in $H$ if and only if it is $\C$-closed in $G$. Note that if $\proC(H)$ is a restriction of $\proC(G)$ then $H$ is $\C$-closed in $G$ as $H$ is $\C$-closed in $H$ by definition. 

\begin{lemma}
    \label{lemma:restriction}
    Let $G$ be a group and let $H \leq G$ be a subgroup. Then $\proC(H)$ is a restriction of $\proC(G)$ if and only if every $N \in \NC(H)$ is $\C$-closed in $G$.
\end{lemma}
\begin{proof}
    Suppose that $\proC(H)$ is a restriction of $\proC(G)$ and let $N \in \NC(H)$ be arbitrary. Clearly, $N$ is $\C$-closed in $H$ and thus it is $\C$-closed in $G$.
    
    Now suppose that every $N \in \NC(H)$ is $\C$-closed in $G$ and let $X \subseteq H$ be $\C$-closed in $H$. Obviously, for every $g \in H$ there is some $N_g \in \NC(H)$ such that $g \notin XN_g$. As $|H:N_g| < \infty$ we see that there are $g_1, \dots, g_n \in X$ such that $XN_g = \bigcup_{i=1}^n g_i N_g$. This means that $XN_g$ is a finite union of sets $\C$-closed in $G$ and therefore it is $\C$-closed in $G$. In particular, we see that $X = \cap_{g \in H\setminus X} XN_g$ is an intersection of $\C$-closed sets in $G$, hence it is $\C$-closed in $G$.
\end{proof}

One can easily show the following by using the proof of \cite[Lemma 3.1.5]{rz}
\begin{lemma}
    \label{lemma:retract_restriction}
    Let $G$ be a residually $\C$ group and suppose that $R \leq G$ is a retract. Then $\proC(R)$ is a restriction of $\proC(G)$.
\end{lemma}

\begin{lemma}
    \label{lemma:restriction_product}
    Let $G_1, G_2$ be groups and let $H_1 \leq G_1$, $H_2 \leq G_2$ be their subgroups. Suppose that $\proC(H_1)$ is a restriction of $\proC(G_1)$ and that $\proC(H_2)$ is a restriction of $\proC(G_2)$. Then $\proC(H_1 \times H_2)$ is a restriction of $\proC(G_1 \times G_2)$.
\end{lemma}
\begin{proof}
    Following Lemma \ref{lemma:restriction} we see that $H_1$ is $\C$-closed in $G_1$, similarly $H_2$ is $\C$-closed in $G_2$. Let $N \in \NC(H_1 \times H_2)$ be arbitrary. Set $N_1 = N \cap H_1$ and $N_2 = N \cap H_2$. As the class $\C$ is closed under taking subgroups, we see that $N_1 \in \NC(H_1)$ and $N_2 \in \NC(H_2)$. As $\C$ is closed under taking direct products, $N_1 \times N_2 \in \NC(H_1 \times H_2)$ and therefore $N_1 \times N_2$ is of finite index in $H_1 \times H_2$ by Lemma \ref{lemma:C-open}. Consequently, $N_1 \times N_2$ is of finite index in $N$. 
    
    As $N_1$ is $\C$-closed in $H_1$ and $\proC(H_1)$ is a restriction of $G_1$, we see that $N_1$ is $\C$-closed in $G_1$. Similarly, $N_2$ is $\C$-closed in $G_2$. Using Remark \ref{remark:useful} we see that $N_1 \times N_2$ is $\C$-closed in $G_1 \times G_2$. As $N_1 \times N_2$ is a subgroup of finite index in $N$ we see that $N$ is $\C$-closed in $G_1 \times G_2$. It follows by Lemma \ref{lemma:restriction} that $\proC(H_1 \times H_2)$ is a restriction of $\proC(G_1 \times G_2)$.
\end{proof}
\section{Graph products}
\label{section:graph_products}
    We recall here some terminology and facts about graph products that will be used in
this paper.
Let $G = \Gamma\mathcal{G}$ be a graph product. Every element $g \in G$ can be obtained as a product of a sequence $W \equiv (g_1, g_2, \dots , g_n)$, where each $g_i$ belongs to some $G_{v_i} \in \mathcal{G}$. We say that $W$ is a \emph{word} in $G$ and that the elements $g_i$ are its \emph{syllables}. The \emph{length} of a word is the number of
its syllables, and it is denoted by $\lvert W\rvert$.

Transformations of the three following types can be defined on words in graph products:
\begin{enumerate}
	\item[(T1)]	remove the syllable $g_i$ if $g_i =e_{G_v}$, where $v \in V$ and $g_i \in G_v$,
	\item[(T2)]	remove two consecutive syllables $g_i, g_{i+1}$ belonging to the same vertex group $G_v$ and replace them by the single syllable $g_i g_{i+1} \in G_v$,
	\item[(T3)] interchange two consecutive syllables $g_i \in G_u$ and $g_{i+1} \in G_v$ if $\{u, v\} \in E$.
\end{enumerate}
The last transformation is also called \emph{syllable shuffling}. Note that transformations (T1) and (T2) decrease the length of a word, whereas transformation (T3) preserves it. Thus, applying finitely many of these transformations to a word $W$, we can obtain a word $W'$ which is of minimal length and that represents the same element in $G$.

For $1 \leq i < j \leq n$, we say that syllables $g_i, g_j$ can be \emph{joined together} if they belong to the same vertex group and \lq everything in between commutes with them\rq. More formally: $g_i, g_j \in G_v$ for some $v \in V$ and for all $i < k < j$ we have that $g_k \in G_{v_k}$ for some $v_k\in\link(v) := \{u \in V \mid \{u,v\}\in E\}$. 
In this case the words 
$$W \equiv (g_1, \dots, g_{i-1}, g_i, g_{i+1}, \dots, g_{j-1}, g_j, g_{j+1}, \dots, g_n)$$ and $$W' \equiv (g_1, \dots, g_{i-1}, g_i g_j, g_{i+1}, \dots, g_{j-1}, g_{j+1}, \dots, g_{n})$$ represent the same group element in $G$, and the length of the word $W'$ is strictly shorter than $W$.

We say that a word $W \equiv (g_1, g_2, \dots , g_n)$ is \emph{reduced} if it is either the empty word, or if $g_i \neq e$ for all $i$ and no two distinct syllables can be joined together. As it turns out, the notion of being reduced and the notion of being of minimal length coincide, as it was proved by Green \cite[Theorem~3.9]{green}:
\begin{theorem}[Normal Form Theorem]
	\label{nft}
Every element $g$ of a graph product $G$ can be represented by a reduced word. Moreover, if two reduced words $W, W'$ represent the same element in the group $G$, then $W$ can be obtained from $W'$ by a finite sequence of syllable shufflings. In particular, the length of a reduced word is minimal among all words representing $g$, and a reduced word represents the trivial element if and only if it is the empty word.
\end{theorem}

Thanks to Theorem \ref{nft} the following are well defined. Let $g \in G$ and let $W \equiv (g_1,\dots, g_n)$ be a reduced word representing $g$. We define the \emph{length} of $g$ in $G$ to be $|g| = n$
and the \emph{support} of $g$ in $G$ to be
 \begin{displaymath}
    \supp(g) = \{v \in V | \exists i \in \{1,\dots, n\} \mbox{ such that } g_i \in G_v\setminus \{e\}\}. 
 \end{displaymath}
We define $\FL(g) \subseteq V\Gamma$ as the set of all $v \in V\Gamma$ such that there is a reduced word $W$ that represents the element $g$ and starts with a syllable from $G_v$. Similarly, we define $\LL(g) \subseteq V\Gamma$ as the set of all $v \in V\Gamma$ such that there is a reduced word $W$ that represents the element $g$ and ends with a syllable from $G_v$. Note that $\FL(g) = \LL(g^{-1})$.

Let $x, y \in G$ and let $W_x \equiv (x_1, \dots ,x_n),W_y \equiv (y_1 \dots, y_m)$ be reduced expressions for $x$ and $y$, respectively.
We say that the product $xy$ is a \emph{reduced product} if the word $(x_1, \dots, x_n, y_1, \dots, y_m)$ is reduced. Obviously, $xy$ is a reduced product if and only if $|xy| = |x| + |y|$ or, equivalently, if $\LL(x) \cap \FL(y) = \emptyset$. We can naturally extend this definition: for $g_1, \dots, g_n \in G$ we say that the product $g_1 \dots g_n$ is reduced if $|g_1 \dots g_n| = |g_1| + \dots + |g_n|$.

Every subset of vertices $X \subseteq V\Gamma$ induces a full subgraph $\Gamma_X$ of the graph $\Gamma$. Let $G_X$ be the subgroup of $G$ generated by the vertex groups corresponding to the vertices contained in $X$. Subgroups of $G$ that can be obtained in such way are called \emph{full subgroups} of $G$; according to standard convention, $G_{\emptyset} = \{e\}$. We say that $G_X$ is a \emph{proper full subgroup} if $G_X \neq G$. 

Using the normal form theorem it can be seen that $G_X$ is naturally isomorphic to the graph product of the family $\mathcal{G}_X = \{G_v \mid v \in X\}$ with respect to the full subgraph $\Gamma_X$. 
For these subgroups, there exists a canonical retraction $\rho_X \colon G \to G_X$ defined on the standard generators of $G$ as follows:
\begin{displaymath}
	\rho_X(g) = \begin{cases}
					g &\mbox{ if }g \in G_v \mbox{ for some } v \in X,\\
					1 &\mbox{ otherwise.} 
				\end{cases}
\end{displaymath}
We will often abuse the notation and sometimes consider the retraction $\rho_X$ as a surjective homomorphism $\rho_X \colon G \to G_X$, and sometimes as an endomorphism $\rho_X \colon G \to G$. In that case writing $\rho_{X} \circ \rho_{Y}$, where $Y \subseteq V\Gamma$, makes sense.

Let $A,B \subseteq V\Gamma$ be arbitrary, $G_A, G_B \leq G$ be the corresponding full subgroups of $G$, and let $\rho_A, \rho_B$ be the corresponding retractions. Then $\rho_A$ and $\rho_B$ commute: $\rho_A \circ \rho_B = \rho_B \circ \rho_A$. It follows that $G_A \cap G_B = G_{A \cap B}$ and $\rho_A \circ \rho_B = \rho_{A \cap B}$.

\subsection{Cyclically reduced elements}
 Let $g\in \Gamma\mathcal{G}$ be an element of a graph product, and let $W\equiv (g_1, \dots,g_n)$ be a reduced word representing it. We say that $g$ is \emph{cyclically reduced} if all cyclic permutations $(g_{i+1},\dots,g_n,g_1,\dots,g_i)$ of $W$, for $i=1,\dots, n-1$, are reduced words. 
In view of \cite[Lemma 3.8]{mf} this definition is well posed, because it is independent of the choice of the reduced word $W$ representing the element $g$. In fact, \cite[Lemma 3.8]{mf} guarantees that every element is conjugate to some cyclically reduced element.

We define the \emph{essential support} of an element $g\in G$, denoted by $\esupp(g)$, to be the support $\supp(g')$ of some cyclically reduced element $g'\in g^G$. Note that if $g', g'' \in g^G$ are cyclically reduced, then $\supp(g') = \supp(g'')$ by \cite[Lemma 3.12]{mf}. It follows that the essential support is well defined, that is it does not depend on the choice of the cyclically reduced word $g'\in g^G$.

The following lemma shows that cyclically reduced elements are the elements of minimal length in their conjugacy class. The proof combines arguments of the proofs of \cite[Lemma 3.12]{mf} and \cite[Theorem B]{bf}.
\begin{lemma}
    \label{lemma:cr_are_small}
   Let $x \in \Gamma\mathcal{G}$ be cyclically reduced. For every $y \in x^G$ we have that $|x| \leq |y|$ and $\supp(x) \subseteq \supp(y)$.
\end{lemma}
\begin{proof}
    We suppose that $x$ is fixed and proceed by induction on $|y|$. The base case is trivial: if $|y| = 0$, then $x = e = y$ and the statement holds.
    
    Now suppose that the statement holds for all $y' \in x^G$ with $|y'| \leq n$, and let $y \in x^G$ be such that $|y| = n + 1$. If $y$ is cyclically reduced then $|x| = |y|$ and $\supp(x) = \supp(y)$ by \cite[Lemma 3.12]{mf}. Now, suppose that $y$ is not cyclically reduced. Suppose that $W = (y_1, \dots,y_n, y_{n+1})$ is a reduced expression for $y$. Without loss of generality, we may assume that the syllables $y_1$ and $y_{n+1}$ belong to the same vertex group. Consider
    \begin{displaymath}
        y' = y_{n+1} y y_{n+1}^{-1} = (y_{n+1} y_1) y_2 \dots y_n \in x^G,
    \end{displaymath}
so that $|y'| < |y|$ and $\supp(y') \subseteq \supp(y)$. Since $|y'| \leq n$, by the induction hypothesis we see that $|x| \leq |y'|$ and $\supp(x) \subseteq \supp(y')$. Thus, the proof is complete.
\end{proof}

As an immediate consequence of Lemma \ref{lemma:cr_are_small} we derive the following nice characterisation of cyclically reduced elements.
\begin{corollary}
    \label{corollary:cr_are_minimal}
    Let $g \in \Gamma\mathcal{G}$ be arbitrary. Then $g$ is cyclically reduced if and only if $|g| \leq |f|$ for all $f \in g^G$.
\end{corollary}

\section{Amalgams over retracts}
\label{section:amalgams}
Suppose that $G_A = K_A \rtimes R$, $G_A =K_A \rtimes R$ are semidirect products and let $\rho_A \colon G_A \to R$, $\rho_B \colon G_B \to R$ be the corresponding retractions. We say that $R$ is a \emph{common retract} for $G_A$ and $G_B$. Consider the amalgamated free product $G = G_A \ast_R G_B$,  where the amalgamation is taken along the identity map $\id_R\colon R\to R$. We say that $G$ is an \emph{amalgam over a retract}, and we have that
\begin{equation}\label{eq_normal_form}
G=    (K_A \rtimes R) \ast_R (K_B \rtimes R) \cong (K_A \ast K_B) \rtimes R.
\end{equation}
From equation \eqref{eq_normal_form}, we see that every $g \in G$ can be expressed as a product $g = k_1 \dots k_m r$ for some elements $k_1, \dots, k_m \in K_A \cup K_B$ and $r \in R$. We say that this expression is \emph{reduced} if $k_i \neq e$ for all $i = 1, \dots, m$ and $k_i, k_{i+1}$ do not belong to the same factor for $i = 1, \dots, m-1$. 
If $m = 0$ and $r = e$, we say that the expression is \emph{trivial}, otherwise it is \emph{non-trivial}.

Combining the standard Normal Form Theorem for amalgamated products (see \cite[Theorem 4.4]{mks}) together with equation \eqref{eq_normal_form} we obtain the following lemma.
\begin{lemma}[Normal form theorem for amalgams over retracts]
    \label{lemma:nft for amalgams over retracts}
    Let $G_A = K_A \rtimes R$, $G_B = K_B \rtimes R$ be groups with a common retract and let $G = G_A \ast_R G_B$ be the corresponding amalgam over the retract $R$. For every $g \in G$ the corresponding reduced expression is given uniquely, and a reduced expression represents the trivial element in $G$ if and only if it is trivial.  
\end{lemma}

Let $g \in G$ be arbitrary. With $ \lvert g\rvert_*$ we denote the free-product length of the factor of $g$ belonging to $K_1 \ast K_2$, that is, if $g=k_1\dots k_nr$ is a reduced expression for the element $g$, then $\lvert g\rvert_*=n$. 
This is the word length in $K_1\ast K_2$ with respect to the (potentially infinite) generating set $K_1\cup K_2$.

Note that if $|g|_* > 0$ then $g \neq e$, but the opposite implication does not hold: $|r|_* = 0$ for every $r \in R$.

\subsection{Graph products as amalgams over retracts}
Let $\Gamma$ be a graph and $C\subseteq V\Gamma$. We say that the subset $C$ is \emph{separating} if the induced subgraph $\Gamma_{V\Gamma\setminus C}$ has at least two connected components.
Therefore, the graph $\Gamma$ contains a separating set if and only if $\Gamma$ contains a pair of distinct vertices $u,v \in V\Gamma$ such that $\{u,v\} \not \in E\Gamma$, i.e. if $\Gamma$ is not complete. Note that, by definition, if the graph $\Gamma$ is disconnected then the empty set $\emptyset \subseteq V\Gamma$ is a separating set in $\Gamma$. In fact, $\emptyset$ is separating if and only if $\Gamma$ is disconnected.
\begin{lemma}
    \label{lemma:split}
    Let $\Gamma$ be a graph and let $\mathcal{G} = \{G_v \mid v \in V\Gamma\}$ be a family of non-trivial groups. If $\Gamma$ contains a separating set, then $G = \Gamma \mathcal{G}$ properly splits as an amalgam over a retract, where the factors (and the amalgamated subgroup) are full subgroups.
\end{lemma}
\begin{proof}
Suppose that $C \subset V\Gamma$ is a (potentially empty) separating set of vertices, so that the graph $\Gamma \setminus C$ has at least two connected components. Write $V\Gamma \setminus S=A'\sqcup B'$, where $A'$ is one of these connected components, and $B'$ consists of all remaining vertices.
It follows that $G\cong G = G_A \ast_{G_C} G_B$, where $A=A'\cup C$ and $B=B'\cup C$.

The subgroups $G_A$, $G_B$ and $G_C$ are retracts of $G$ because they all are full subgroups. For the same reason, $G_C$ is a retract of both $G_A$ and $G_B$.
\end{proof}
Note that the choice of the splitting in Lemma \ref{lemma:split} is not unique, and depends both on the choice of the separating subset $C$, and on how to express the set $V\Gamma\setminus C$ as the union of two sets given by its connected components.

For a vertex $v \in V\Gamma$ we define the \emph{link} $\link(v)$ to be the set of vertices adjacent to $v$ in$V\Gamma$, and the \emph{star} to be
$\star(v)=\link(v)\cup\{v\}$. For a subset $A\subseteq V\Gamma$, we define $\star(A)$ to be $\bigcap_{v\in A}\star(v)$.

Separating subsets can be obtained using links of vertices: if $v\in V\Gamma$ is a vertex that is not connected by an edge to every other vertex of $\Gamma$,  i.e. if $\link(v) \neq V\Gamma \setminus \{v\}$, then $\link(v)$ is a separating subset, and the induced splitting is
\[G\cong G_{V\Gamma\setminus\{v\}}\ast_{G_{\link(v)}}\bigl(G_{\link(v)}\times G_v\bigr).\]
We did not use this fact in the proof of Lemma \ref{lemma:split}, because it might happen that the minimal separating subset in a graph cannot be expressed as the link of a vertex, as for instance in the following graph:
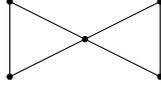
\begin{equation*}
\begin{tikzpicture}
\draw (0,0) --(1,.5) --(1,-.5) --(0,0) --(-1,.5) --(-1,-.5) --(0,0);
\draw[fill] (0,0) circle [radius=0.035];
\draw[fill] (1,.5) circle [radius=0.035];
\draw[fill] (1,-.5) circle [radius=0.035];
\draw[fill] (-1,.5) circle [radius=0.035];
\draw[fill] (-1,-.5) circle [radius=0.035];
\end{tikzpicture}
\end{equation*}

From Lemma \ref{lemma:split} one deduces Theorem \ref{general:GP_CSS}, as a graph product is either a direct product of vertex groups (when the underlying graph is complete), or it splits as an amalgam over a common retract.
\begin{thmBB}\emph{
Let $\mathcal{P}$ be a property of groups that is preserved under finite direct products and under amalgamation over retracts. Then $\mathcal P$ is preserved under finite graph products.}
\end{thmBB}
\begin{proof}
    Let $\Gamma=(V\Gamma,E\Gamma)$ be a finite graph, and $\mathcal{G}$ be a family of groups that have $\mathcal{P}$, indexed by the vertex set of $\Gamma$.
    If $\Gamma$ is a complete graph, then $\Gamma\mathcal{G}$ is a finite direct product of groups with property $\mathcal{P}$, and therefore it satisfies property $\mathcal P$.
    
    If $\Gamma$ is not complete, then it contains a separating set $S \subset V\Gamma$.
    Therefore, by Lemma \ref{lemma:split}, $\Gamma\mathcal{G}$ is an amalgam over a common retract, $\Gamma\mathcal{G}= G_1 \ast_{G_S} G_2$, for some proper full subgroups $G_1, G_2, G_S \leq \Gamma\mathcal{G}$.
    By induction hypothesis the groups $G_1, G_2$ satisfy property $\mathcal{P}$. Therefore, $\Gamma\mathcal{G}$ satisfies $\mathcal{P}$.
\end{proof}

\subsection{Cyclically reduced elements}
Let $g \in (K_1\ast K_2)\rtimes R$ and let $g = k_1 \dots k_m r$ 
 be a reduced expression for $g$. With a slight abuse of notation, we say that $g$ is \emph{cyclically reduced} if either one of the following is true:
\begin{itemize}
    \item [(i)] $m \in \{0, 1\}$,
    \item [(ii)] $m \geq 2$ and $k_1, k_m$ do not belong to the same factor.
\end{itemize}
That is, we say that $g$ is cyclically reduced if $gr^{-1}\in K_1\ast K_2$ is cyclically reduced in the sense of the free product $K_1 \ast K_2$ (see \cite[Section 4.1]{mks}).
Note that if $|g|_*$ is even then $g$ is necessarily cyclically reduced. If $|g|_* = 2l + 1$ for some $l \in \mathbb{N}$, then $g$ is cyclically reduced if and only 
if $|g|_* = 1$.

If $g = k_1 \dots k_m r$ is a reduced expression for an element $g \in (K_1\ast K_2)\rtimes R$, we say that $c \in K_1 \ast K_2$ is a \emph{prefix} of $g$ if $c = k_1 \dots k_{m'}$ for $m' \leq m$.
\begin{lemma}\label{lemma:amalgam_conjugator_length}
    Let $g \in (K_1\ast K_2)\rtimes R$ be not cyclically reduced with $\lvert g\rvert_* = 2l + 1$, where $l \geq 1$. There exists a prefix $c$ of $g$ such that $c^{-1} g c$ is cyclically reduced and $\lvert c\rvert_* \leq l$. 
\end{lemma}
\begin{proof}
    Let $g = k_1 \dots k_{2l+1}r$ 
    be the reduced expression for $g$. By assumption, $k_1$ and $k_{2l+1}$ belong to the same factor, and moreover $k_i$ and $k_{2l+2-i}$ belong to the same factor, for all $i =1,\dots, l$.

    Suppose that that there exists $0< l' < l$ such that $(k_{2l+2-l'}k_{l'}^r) \neq 1$,
    where $k_{l'}^r=rk_{l'}r^{-1}$ by definition.
    Let $l'$ be the smallest natural possible with this property, and set $c = k_1 \dots k_{l'}$. Then 
    \begin{equation}\label{reduced_for_Cgc}
    \begin{split}
        c^{-1}gc &=(k_1\dots k_{l'})^{-1}(k_1\dots k_{2l+1}r)  (k_1\dots k_{l'}) \\
        &= k_{l'+1} \dots k_{2l + 1 - l' } (k_{2l + 2 - l'}k_{l'}^r) r.
    \end{split}\end{equation}
    As the element $(k_{2l+2-l'}k_{l'}^r)$ is not trivial, the expression of equation \eqref{reduced_for_Cgc} is reduced. 
    Moreover, the elements $k_{l'+1}$ and $k_{l'}$ (and therefore $k_{l'+1}$ and $(k_{2l-l'}k_{l'}^r)$) belong to different factors. We therefore see that the element $c^{-1}gc$ is cyclically reduced.
    
    If no such $l'$ exists, set $c = k_1 \dots k_l$. We have that
    \begin{equation*}
    c^{-1} g c = (k_1\dots k_l)^{-1}(k_1\dots k_{2l+1}r)  (k_1\dots k_l)  = k_{l+1} r   
    \end{equation*}
    is cyclically reduced.
\end{proof}

\begin{lemma}
    \label{lemma:amalgam_CR_powers}
    For $i=1,2$ let $G_i = K_i \rtimes R$ be torsion-free groups, consider the amalgam
    $G = G_1 \ast_R G_2 \cong (K_1 \ast K_2)\rtimes R$, and
    suppose that for every $g \in G_i \setminus R$ we have $\langle g \rangle \cap R = \{e\}$. 
    Let $g \in G$ and $n \in \mathbb{Z}\setminus\{0\}$ be arbitrary. Then $g^n$ is cyclically reduced if and only if $g$ is cyclically reduced. 
    
    Furthermore, in this case, we have that $\lvert g^n\rvert_* = 1$ if and only if $\lvert g\rvert_* = 1$, and $\lvert g^n\rvert_* = \lvert n\rvert \cdot\lvert g\rvert_*$ otherwise.
\end{lemma}
\begin{proof}
    Let $g = k_1 \dots k_m r$ be the reduced expression for $g$ and suppose that $g$ is cyclically reduced. There are two cases to consider: either $m \in \{0,1\}$, or $m = 2l$ for some $l \geq 1$.
    
    If $m \in \{0,1\}$, then $g$ belongs to one of the factors $G_i$, and therefore $g^n$ is cyclically reduced as well.
    
    Suppose that $m = 2l$ for some $l \geq 1$. We have that
    \begin{equation}\label{eq_cocky_bit}
        g^n = k_1 \dots k_{2l} k_1^r \dots k_{2l}^r \dots k_1^{r^{n-1}} \dots k_{2l}^{r^{n-1}} r^n
    \end{equation}
    is the reduced expression for $g^n$, and therefore $g^n$ is cyclically reduced.
    
    \smallskip
    Now assume that $g$ is not cyclically reduced, so that $m=2l+1$ for some $l\geq 1$. Following Lemma \ref{lemma:amalgam_conjugator_length}, $g$ has a prefix $c \in K_1 \ast K_2$ such that $c^{-1} g c$ is cyclically reduced and $|c| \leq l$. There are two subcases to distinguish: $l'= l$ or $l' < l$, where $l'$ is as in Lemma \ref{lemma:amalgam_conjugator_length}.
    
    If $l' = l$ then $c^{-1} g c = k_{l+1} r$ belongs to one of the factors: without loss of generality let us assume that $k_{l+1} \in K_1$ and consequently $c^{-1}gc \in G_1$.

    Denote $k = k_{l+1} k_{l+1}^r \dots k_{l+1}^{r^{n-1}} \in K_1$. Then we have
    \begin{align*}
        g^n &= cc^{-1} g^n cc^{-1} = c(c^{-1} g c)^nc^{-1}\\
            &= (k_1 \dots k_l)(k_{l+1}r)^n(k_1 \dots k_l)^{-1}\\
            &= (k_1 \dots k_l)(k r^n)(k_1 \dots k_l)^{-1}\\
            &= k_1 \dots k_l k (k_l^{-1})^{r^n} \dots (k_1^{-1})^{r^n}r^n.
    \end{align*}
    If $k = e$ is trivial then $(c^{-1}gc)^n = r^n \in R$. As $c^{-1}gc$ has infinite order, because the groups $G_i$ are torsion-free, it follows that $\langle c^{-1}gc \rangle \cap R \neq \{e\}$, which is a contradiction with the assumptions, as $c^{-1}gc \notin R$. By construction, as the element $g = k_1 \dots k_m r$ is reduced, the elements $k_l$ and $k$ belong to different factors, and similarly for $(k_l^{-1})^{r^n}$ and $k$. It follows that $g^n = k_1 \dots k_l \overline{k} (k_l^{-1})^{r^n}\dots (k_1^{-1})^{r^n} r^n$ is the reduced expression for $g^n$ and therefore $g^n$ is not cyclically reduced, as $k_1$ and $(k_1^{-1})^{r^n}$ belong to the same factor.
    
    \smallskip
    If $l'<l$, then the expression
    \[c^{-1}gc = k_{l'+1} \dots k_{2l +1 - l' } (k_{2l +2 - l'}k_{l'}^r) r\]
    is reduced, as can be seen in equation \eqref{reduced_for_Cgc}. 
    Let $w:= k_{l'+1} \dots k_{2l - l' - 1} (k_{2l - l'}k_{l'}^r)$ be an element of  $K_1 \ast K_2$. It follows that
    \begin{equation}\label{eq_case_lsmallerthanlprime}
        g^n = c(c^{-1}gc)^nc^{-1} = c(k r)^n c^{-1} = c w w^r \dots w^{r^{n-1}} (c^{-1})^{r^n}r^n.
    \end{equation}
    Note that the last letter of $c$ and the fist letter of $w$ belong to different factors, and the same is true for the first and last letter of $w$.
    
    It follows that, up to replacing all occurrences of $w$ with its expansion in terms of the elements $k_i$, the expression for $g^n$ given in equation \eqref{eq_case_lsmallerthanlprime} is reduced. The last letter of $w^{r^{n-1}}$ is $k_{2l + 2 - l'}^{r^{n-1}}k_{l'}^{r^n}$ and the first letter of $(c^{-1})^{r^n}$ is $(k_l^{-1})^{r^n}$. Multiplying those two we get $k_{2l - l'}^{r^{n-1}}$. It then follows that $g^n$ is not cyclically reduced.
    
    The last part of the statement follows from the reduced expression of equation \eqref{eq_cocky_bit}.
\end{proof}

We spell out the following fact, which was just proved in Lemma \ref{lemma:amalgam_CR_powers}:
\begin{corollary}\label{cor:spelt_out}
Let the groups $G_i$ and $G$ be as in Lemma \ref{lemma:amalgam_CR_powers}, and $g=k_1\dots k_mr$ be the reduced expression for the cyclically reduced element $g\in G$, with $m=\lvert g\rvert_*>1$. Then 
\[g^n=k_1\dots k_m k_1^r\dots k_m^r\dots k_1^{r^{n-1}}\dots k_m^{r^{n-1}}r^n\]
is the reduced expression of the element $g^n$, for all $n\geqslant 1$.
\end{corollary}
\section{Separating cyclic subgroups of graph products in the profinite topology}
\label{section:profinite_case}
The following result is proved by Bobrovskii and Sokolov in \cite{bobr}.
\begin{theorem}
    \label{theorem:bobrovskii}
    Let $G = G_1 \ast_R G_2$ be an amalgam over a common retract, let $g \in G$ be arbitrary, and suppose that $G_1$ and $G_2$ are residually finite. Then $\langle g \rangle$ is not separable in $G$ if and only if $g$ is conjugate to some $g_i \in G_i$, where $i \in \{1,2\}$, and $\langle g_i \rangle$ is not separable in $G_i$.
\end{theorem}

We will use the following lemma to shorten our proofs.
\begin{lemma}
    \label{remark:shortening}
    Let $G = \Gamma \mathcal{G}$ be a graph product of residually finite groups and let $g \in G$ be arbitrary. The cyclic subgroup $\langle g \rangle \leq G$ is separable in $G$ if and only if it is separable in $G_S$, where $S = \supp(g)$. Furthermore, $\langle g \rangle$ is separable in $G$ if and only if $\langle g' \rangle$ is separable in $G$ for some (and hence for all) $g' \in g^G$.
\end{lemma}
\begin{proof}
    Graph products of residually finite groups are residually finite by \cite[Corollary 5.4]{green}, hence $G$ is residually finite. Since $G_S$ is a retract of $G$, its profinite topology  $\PT(G_S)$ is a restriction of $\PT(G)$ by Lemma \ref{lemma:retract_restriction}. Therefore $\langle g\rangle$ is separable in $G$ if and only if it is separable in $G_S$.
    
    Now let $\phi \in \Inn(g)$ be an inner automorphism of $G$ and let $\phi(g) = g'$. Clearly, $\langle g' \rangle = \langle \phi(g) \rangle = \phi\left( \langle g \rangle \right)$. Lemma \ref{lemma:continuous} implies that $\phi$ is a homeomorphism of $\PT(G)$, hence $\langle g \rangle$ is separable in $G$ if and only if $\langle g' \rangle$ is separable in $G$.
\end{proof}

\begin{lemma}
    \label{lemma:gp_separable_elements}
    Let $G = \Gamma \mathcal{G}$ be a graph product of residually finite groups and let $g \in G$ be a cyclically reduced element such that the full subgraph $\Gamma_S$ contains a separating subset, where $S = \supp(g)$. Then the cyclic subgroup $\langle g \rangle \leq G$ is separable in $G$.
\end{lemma}
\begin{proof}
    
    Using Lemma \ref{remark:shortening}, we may assume that $S = V\Gamma$ and subsequently $\Gamma = \Gamma_S$. As $\Gamma$ is not complete, there exist a pair of vertices $u,v \in V\Gamma$ and a separating set $C \subseteq V\Gamma$ such that $u, v$ lie in distinct connected components of $\Gamma \setminus C$, say $\Gamma_{A'}$\ and $\Gamma_{B'}$, for $A',B'\subseteq V\Gamma\setminus C$. Without loss of generality we may assume that $V\Gamma = A' \cup B' \cup C$. Set $A = A' \cup C$ and $B = B' \cup C$. As mentioned in Lemma \ref{lemma:split}, $G$ splits as an amalgam over a common retract $G = G_A \ast_{G_C} G_B$. 
    By Lemma \ref{lemma:nft for amalgams over retracts}, the element $g$ can be written as \begin{displaymath}
        g = a_1 b_1 \dots a_n b_n r,
    \end{displaymath}
    for some uniquely given $a_1, \dots a_n \in \ker(\rho_A)$, $b_1, \dots b_n \in \ker(\rho_B)$ and $r \in G_C$, where $\rho_A \colon G_A \to G_C$ and $\rho_B \colon G_B \to G_C$ are the canonical retractions. 
    
    As the element $g$ is cyclically reduced in the sense of graph products and $\supp(g)=V\Gamma$, $g$ cannot be conjugated to an element of any of the two proper subgroups $G_A$ or $G_B$, by Lemma \ref{lemma:cr_are_small}.
    By Theorem~\ref{theorem:bobrovskii}, it must be that $\langle g \rangle$ is separable in $G$.
\end{proof}

Combining the lemma above with Lemma \ref{remark:shortening}, we immediately get the following.
\begin{corollary}
    \label{corollary:separable_elements}
    Let $\Gamma$ be a graph, $\mathcal{G} = \{G_v \mid v \in V\Gamma\}$ be a family of residually finite groups and $G = \Gamma \mathcal{G}$ be the corresponding graph product. Suppose that $g \in G$ is an arbitrary element such that $\Gamma_S$ contains a separating subset, where $S = \esupp(g)$. Then the cyclic subgroup $\langle g \rangle \leq G$ is separable in $G$.
\end{corollary}

\begin{lemma}
    \label{lemma:restriction_cyclic}
    Let $G$ be a CSS group and let $C \leq G$ be an infinite cyclic subgroup. Then $\PT(C)$ is a restriction of $\PT(G)$.
\end{lemma}
\begin{proof}
    Let $N \leq C$ be open. As $C$ is cyclic, by necessity, $N$ is cyclic as well. By cyclic subgroup separability of $G$, the subgroup $N$ is closed in $G$. Using Lemma \ref{lemma:restriction} we get the result. 
\end{proof}

The following lemma can be seen as an slight strengthening of \cite[Proposition 4.1]{quasipotency}, where it is shown that the direct product of quasi-potent CSS groups is again CSS. Using a slightly more topological approach, we show that quasi-potency is not necessary. The idea of using restrictions of profinite topologies was suggested to the authors by Ashot Minasyan, a suggestion for which we are very grateful.
\begin{lemma}
    \label{lemma:direct_CSS}
    The class of CSS groups is closed under forming direct products.
\end{lemma}
\begin{proof}
    Let $G_1, G_2$ be CSS groups and let $C \leq G_1 \times G_2$ be cyclic. Let $g = (g_1, g_2) \in G_1 \times G_2$ be a generator of $C$. Set $C_1 = \langle g_1 \rangle \leq G_1$ and $C_2 = \langle g_2 \rangle \leq G_2$. Using Lemma \ref{lemma:restriction_cyclic} we see that $\PT(C_1)$ is a restriction of $\PT(G_1)$ and $\proC(C_2)$ is a restriction of $\PT(G_2)$. It follows by Lemma \ref{lemma:restriction_product} that $\PT(C_1 \times C_2)$ is a restriction of $\PT(G_1 \times G_2)$. Notice that $C_1 \times C_2$ is finitely generated abelian, hence it is LERF. This means that $C$ is closed in $C_1 \times C_2$ and hence $C$ is closed in $G_1 \times G_2$.
\end{proof}
Lemma \ref{lemma:direct_CSS} was already proved combinatorially in \cite[Theorem 4]{stebe}. We believe that our proof is bit more accessible and moreover, under additional assumptions, it can be generalised to pro-$p$ topologies (see Lemma \ref{lemma:direct_pCSS}).

Theorem \ref{proposition:GP_CSS} follows from Theorem \ref{general:GP_CSS}, noting that its hypotheses are satisfied in view of Lemma \ref{lemma:direct_CSS} and of Corollary \ref{corollary:separable_elements}.

\section{Unique roots}
\label{section:unique_roots}
\begin{definition}[{\bf Unique roots}]\label{def:unique_roots}
Let $G$ be a group, and $g\in G$ be an element.
We say that an element $r\in G$ is a \emph{root} of $g$ if there is a positive integer $n \in \mathbb{N}$ such that $r^n = g$ in $G$.
We say that $g \in G$ has \emph{unique roots} if the equation $x^n = g$ has at most one solution for every $n \in \mathbb{N}$, i.e. for every $x,y \in G$ and every $n \in \mathbb{N}$ the equality $x^n = g =y^n$ implies $x = y$. A group $G$ is said to have the \emph{Unique Root property} if every $g \in G$ has unique roots.

We will use $\U$ to denote the class of all groups with Unique Root property.
\end{definition}
As inverses are unique, replacing natural numbers by integers in the definition does not change the notion.
Moreover, if a group has non-trivial torsion elements, then it does \emph{not} have unique roots.

The aim of this section is to establish that the class $\U$ is closed under taking graph products. We start with a fact that will be used in Proposition~\ref{lemma:urp_amalgams}.
\begin{lemma}
    \label{lemma:urp_condition}
    Let $G$ be a group and let $R \leq G$ be a retract. If $G\in\U$ then for every $g \in G \setminus R$ we have $\langle g \rangle \cap R = \{e \}$.
\end{lemma}
\begin{proof}
    Suppose that there is $g \in G \setminus R$ and a $n \in \mathbb{N}$ such that $g^n \in R \setminus \{e\}$. Let $\rho \colon G \to R$ be the retraction corresponding to $R$ and set $\rho(g) = r \in R$. We see that $g^n = \rho(g^n) = r^n$. However $g \neq r$, contradicting the Unique Root property.
\end{proof}

 \begin{lemma}
    \label{lemma:urp_direct}
    The class $\U$ is closed under taking subgroups and direct products.
 \end{lemma}
 \begin{proof}
    Let $G$ be a group with Unique Root property and suppose that $H \leq G$. Let $x, y \in H$ be arbitrary and suppose that $x^n = y^n$ for some $n \in \mathbb{N}$. As $x, y \in G$ and $G\in \U$, we see that $x = y$. 
 
    For direct products, we prove the statement for a direct product of two groups. The argument applies to any number (finite or not) of direct factors.
    Let $G_1, G_2\in\U$, let $x = (x_1, x_2), y = (y_1, y_2) \in G_1\times G_2$ be arbitrary elements and suppose that $x^n = y^n$ for some $n \in \mathbb{N}$. This means that $(x_1^n, x_2^n) = (y_1^n, y_2^n)$, i.e. $x_1^n = y_1^n$ in $G_1$ and $x_2^n = y_2^n$ in $G_2$. By unique roots, we conclude that $x_1 = y_1$ in $G_1$, and that $x_2 = y_2$ in $G_2$. Therefore $x = y$, and thus the direct product $G_1 \times G_2$ has the Unique Root property.
 \end{proof}

In particular, any retract of a group with the Unique Root property also has the Unique Root property.

In the following proposition we prove that Unique Root property is preserved under taking amalgamations along retracts, and in particular by free products.
\begin{proposition}
    \label{lemma:urp_amalgams}
	The class $\U$ is closed under taking amalgams over retracts.
\end{proposition}
\begin{proof}
	Let $G_1 = K_1 \rtimes R$, $G_2 = K_2 \rtimes R$ be groups in $\U$, let $\rho_1 \colon G_1 \to R$ and $\rho_2 \colon G_2 \to R$ be the corresponding canonical retractions. Set
	\begin{displaymath}
		G = G_1 \ast_R G_2 \cong (K_1 \ast K_2)\rtimes R.
	\end{displaymath}
	and let $\rho \colon G \to R$ be the natural extension of $\rho_1, \rho_2$ to $G$.
	
	Let $x, y \in G$ be arbitrary elements  such that $x^n  = y^n$ for some $n \geq 2$. Let $r_x, r_y \in R$ and $k_x, k_y \in K_1 \ast K_2$ be the uniquely given elements such that $x = k_x r_x$ and $y = k_y r_y$. 
	
	As $\rho(x^n) = \rho(y^n)$, we have that $r_x^n = r_y^n$. The retract $R$ has the Unique Root property by Lemma \ref{lemma:urp_direct}, and therefore we conclude that $r_x = r_y$, which we denote by $r$. We see that
	\begin{align*}
		x^n  &= k_x k_x^r \dots k_x^{r^{n-1}} r^n,\\
		y^n  &= k_y k_y^r \dots k_y^{r^{n-1}} r^n.
	\end{align*}
	As $x^n=y^n$, we obtain that
	\begin{equation}\label{equation_UR_amalgams}
		 k_x k_x^{r} \dots k_x^{r^{n-1}} = k_y k_y^{r} \dots k_y^{r^{n-1}},
	\end{equation}
	and we denote this element by $k$.
    
    Without loss of generality, we can suppose that $x^n$ (and, consequently, also $y^n$) is cyclically reduced. Indeed, if this is not the case, by Lemma \ref{lemma:amalgam_conjugator_length} there exists a prefix $c$ of $x^n$ such that $c^{-1}x^nc$ is cyclically reduced. Therefore, we can replace $x^n$ and $y^n$ with $c^{-1}x^nc$ and $c^{-1}y^nc$ and proceed considering these elements.
    
    By Lemma \ref{lemma:urp_condition}, the groups $G_1$ and $G_2$ satisfy the hypotheses of Lemma~\ref{lemma:amalgam_CR_powers}. Therefore, applying it, we see that both $x$ and $y$ are cyclically reduced.
	Following Lemma \ref{lemma:amalgam_CR_powers}, we see that $\lvert k\rvert_* =  1$ if and only if $\lvert k_x\rvert_* = 1$, if and only if $\lvert k_y\rvert_* = 1$. In this case, it follows that both elements $k_x$ and $k_y$ must belong to the same factor, which without loss of generality we assume to be $K_1$. This means that both $x, y \in G_1$, and from $x^n=y^n$ we conclude that $x = y$, as $G_1\in\U$.
	
	\smallskip
	Now suppose that $\lvert k\rvert_* > 1$. Therefore $\lvert k_x\rvert_*$ and $\lvert k_y\rvert_*$ are greater than one by Lemma \ref{lemma:amalgam_CR_powers}, and moreover $\lvert k\rvert_* = n\lvert k_x\rvert_* = n\lvert k_y\rvert_*$. Suppose that $k_x = k_1 \dots k_m$ is a reduced expression for $k_x$ in $K_1 \ast K_2$, and that $k_y = h_1 \dots h_m$ is a reduced expression for $k_y$, where $m=\lvert k_x\rvert_*=\lvert k_y\rvert_*$. As we assumed $x^n$ to be cyclically reduced, we conclude that both these expressions are cyclically reduced. 
	
	From equation \eqref{equation_UR_amalgams}, $k$ can be expressed as
	\begin{equation*}
	\begin{split}
	   (k_1 \dots k_l) (k_1^r \dots k_l^r) \dots &(k_1^{r^{n-1}} \dots k_l^{r^{n-1}}) = \\ &=(h_1 \dots h_m)( h_1^r \dots h_m^r) \dots (h_1^{r^{n-1}} \dots h_m^{r^{n-1}}).
	\end{split}
	\end{equation*}
	By the normal form theorem for free products (see \cite[Theorem 4.1]{mks} or \cite[Chapter IV, Theorem 1.2]{ls}) we obtain that $k_i = h_i$ for $i = 1, \dots, m$. Hence $k_x = k_y$, and therefore
	\[x =k_xr=k_yr= y.\] 
    Thus $G\in\U$.
\end{proof}

By Lemma \ref{lemma:urp_direct} and Proposition \ref{lemma:urp_amalgams}, applying Theorem \ref{general:GP_CSS} we conclude:
\begin{theorem}
	The class $\U$ is closed under taking graph products.
	\label{lemma:urp_gp}
\end{theorem}

Since $\mathbb{Z}\in\U$, we re-obtain the following corollary, originally proven in \cite[Lemma~6.3]{ashot_raags}.
\begin{corollary}
    Right-angled Artin groups satisfy the Unique Root property.
\end{corollary}
\section{Primitive roots and $p$-isolation}
\label{section:primitive}
	From now, we will consider the class $\C$ to consist of all finite $p$-groups for some prime number $p$. Given a group $G$, we use $\Np(G)$ to denote the set of all co-$p$-finite subgroups of $G$ and we use $\prop(G)$ to denote the $\prop$ topology on $G$. Also, for a subset $X \subseteq G$ we use the term $p$-separable or $p$-closed in $G$ to denote that $X$ is closed in the $\prop$ topology on $G$.
	
	First, let us consider the following example.
	\begin{example}
	\label{example:p_inseparable_pair}
	    Let $G$ be a an arbitrary infinite group, suppose that there is $g_0 \in G$ such that $\ord_G(g_0) = \infty$ and $g \in G$ such that $g_0 = g^q$ for some prime number $q$ distinct from $p$. Note that $g \not\in \langle g_0 \rangle$ and $\langle g_0 \rangle \leq \langle g \rangle$. Assume that $\pi \colon G \twoheadrightarrow Q$ is a surjective homomorphism onto some finite $p$-group $Q$. Then $\langle \pi(g)\rangle$ is a cyclic group of order $p^e$ for some $e \in \mathbb{N}$. As $\pi(g_0) = \pi(g)^q$ and $\gcd(q, p^e) = 1$, we see that $\pi(g_0)$ generates $\langle \pi(g)\rangle$ and therefore $\pi(g) \in \langle\pi(g_0)\rangle = \pi\left(\langle g_0\rangle\right)$. In particular, the cyclic subgroup $\langle g_0 \rangle \leq G$ is not closed in the pro-$p$ topology on $G$. 
	\end{example}
	This example motivates the following definition.
\begin{definition}[{\bf $p$-isolation}]\label{def:p.isolation}
    Let $G$ be a group and let $H \leq G$. We say that $H$ is $p$-\emph{isolated} in $G$ if for every $f \in G$ and every prime number $q$ distinct from $p$ the following holds:
    \begin{displaymath}
        f^q \in H \quad\Rightarrow \quad f \in H. 
    \end{displaymath}
    An element $g\in G$ is said to be $p$-isolated in $G$ if the cyclic subgroup $\langle g\rangle$ is $p$-isolated in $G$. A subgroup $H \leq G$ is said to be \emph{isolated} in $G$ if it is $p$-isolated for every prime $p$.
\end{definition}
    The authors of \cite{bobr} use the term \emph{$p'$-isolated} for the same notion. To ease the notation, we decided to drop the $'$ as there is no chance of confusion.

    Following Example \ref{example:p_inseparable_pair}, we see that being $p$-isolated is a necessary condition for a subgroup to be $p$-separable, hence it makes sense to consider $p$-separability only for $p$-isolated subgroups. However, it was shown in \cite{bardakov} that a non-abelian free group contains a finitely generated subgroup which is $p$-isolated but not $p$-separable for any prime number $p$. 
    
	Therefore, we pose the following definition:
	\begin{definition}[{\bf $p$-cyclic subgroup separability}]
	A group $G$ is \emph{$p$-cyclic subgroup separable} ($p$-CSS) if every $p$-isolated cyclic subgroup of $G$ is $p$-separable in $G$. 
    \end{definition}
    
    The aim of this section is to give a useful description of $p$-isolated cyclic subgroups of groups.
    \begin{lemma}
        \label{lemma:simplified_p-isolation}
        Let $G$ be a group and suppose that $g \in G$ is of infinite order. 
        The cyclic group $\langle g \rangle \leq G$ is $p$-isolated if and only if
        for every $n \in \mathbb{N}$ coprime to $p$ and every $f \in G$
        \begin{displaymath}
            f^n \in \langle g \rangle\quad \Rightarrow\quad f \in \langle g \rangle.
        \end{displaymath}
    \end{lemma}
    \begin{proof}
        Only one implication is non trivial, so
        let $f \in G$ be arbitrary and suppose that $f^n \in \langle g \rangle$ for some $n$ coprime to $p$. Let $n = p_1^{e_1} \dots p_m^{e_m}$ be the prime factorisation of $n$. We will proceed by induction on $N = e_1 + \dots + e_m$. 
        
        If $N = 1$ then $n$ is a prime and the statement holds. Suppose that the statement has been proved for all $n'$ whose sum of exponents in the prime decomposition is less than $N$. Set $f' = f^{p_1}$ and $n' = n/{p_1}$. Note that $n'$ is coprime with $p$. As $(f')^{n'} = f^n \in \langle g \rangle$, we have that $f' \in \langle g \rangle$ by induction hypothesis. Moreover $f' = f^{p_1}$, and therefore $f \in \langle g \rangle$.        
    \end{proof}
    Informally speaking, a subgroup is $p$-isolated if it is closed under taking ``$n$-th roots'' for $n$ coprime with $p$. This informal observation motivates the rest of this section.
    
    \begin{definition}[{\bf Primitive roots, and primitive logarithms}]
    \label{definition:plogs}
    Let $G$ be a group.
    As defined in the previous section, an element $r\in G$ is a root of $g\in G$ if there is a positive integer $k \in \mathbb{N}$ such that $r^k = g$ in $G$. 
    We say that $r$ is a \emph{primitive root} of $g$ in $G$ if such $k$ is maximal possible:  
    \[k = \max\{n \in \mathbb{N} \mid \exists r \in G \colon r^n = g\}.\] 
    We use $\sqrt[G]{g}$ to denote the set of all primitive roots of $g$ in $G$. When the group $G$ has unique roots, we slightly abuse notation and use $\sqrt[G]{g}$ to denote the primitive root of $g$ in $G$.
    
    If $r$ is a primitive root of $g$ in $G$ with corresponding exponent $k \in \mathbb{N}$ so that $r^k = g$, then we say that $k$ is the \emph{primitive logarithm} of $g$ in $G$, and we denote it as $k = \plog_G(g)$. If $\sqrt[G]{g}=\emptyset$, then the primitive logarithm is not defined for $g$.
    \end{definition}
    
    If $g\in G$ has finite order $n$, then $\sqrt[G]{g}$ is empty because $g^{kn}=e_G$ for all $k\in\mathbb{N}$, and therefore there is no maximal. 
    
    However, this is not the only case when an element $g\in G$ might not have a primitive root. 
    Indeed, consider the Baumslag-Solitar group 
    \begin{displaymath}
        G = \BS(1,2)=\langle a,t \| tat^{-1} = a^2 \rangle.
    \end{displaymath}
    From the relation of $G$ one deduces that $(t^{-n} a t^n)^{2^n} = a$ for every $n \in \mathbb{N}$, and therefore the element $a$ has no primitive roots: $\sqrt[G]{a} = \emptyset$.

    \begin{remark}
        \label{remark:plog_conjugacy}
        Let $G$ be a group and let $g \in G$ be an element with primitive root. For any $c \in G$ we have that 
        \begin{equation*}
        \plog_G(cgc^{-1}) = \plog_G(g), \qquad \sqrt[G]{cgc^{-1}} = c \left(\sqrt[G]{g}\right)c^{-1}.
        \end{equation*}
    \end{remark}
    \smallskip
    \noindent
    Consider a group $G$ given by the presentation
    \begin{displaymath}
        \langle x, y \| x^p = y^q\rangle,
    \end{displaymath}
    where $p < q$ are distinct primes. Then the element $x$ is its own (unique) primitive root, that is $\sqrt[G]{x}=\{x\}$, but $\sqrt[G]{x^p} = \{y\}$ and $\plog_G(x^p) = q$. 
    
    This motivates the following definition.
    
    \begin{definition}[{\bf Primitive stability}]
    \label{definition:primitive_stability}
    We say that an element $g \in G$ is \emph{primitively stable} in $G$ if $\sqrt[G]{g}$ is defined and $\plog_G(g^n) = n \cdot \plog_G(g)$ for all $n \in \mathbb{N}$. We say that a group $G$ is \emph{primitively stable} if every $g \in G \setminus \{e\}$ is primitively stable. 
    \end{definition}
    Note that primitively stable groups are necessarily torsion-free.
    We denote by $\Ups$ the class of primitively stable groups with unique roots.

    \smallskip
    The following will not be used during the text, but it provides a nice characterisation for primitively stable elements and provides a comparison to Proposition \ref{lemma:cyclic_radical_criterion}.
    \begin{lemma}
    Let $G$ be a group. For any element $g\in G$ such that $\sqrt[G]{g}\neq\emptyset$ and any natural number $n\geqslant 1$ we have that $\plog_G(g^n) = n \cdot \plog_G(g)$ if and only if $\sqrt[G]{g}\subseteq \sqrt[G]{g^n}$.
    \end{lemma}
    \begin{proof}
    Suppose that  $\plog_G(g^n) = n \cdot \plog_G(g)$, and let $r\in \sqrt[G]{g}$, so that 
    $r^{plog_G(g)}=g$.
    By taking powers, we obtain that $r^{n\cdot plog_G(g)}=g^n$, and the hypothesis implies that $r^{ plog_G(g^n)}=g^n$. This, by definition, means that $r\in\sqrt[G]{g^n}$. Therefore 
    $\sqrt[G]{g}\subseteq \sqrt[G]{g^n}$.
    
    Suppose now that $\sqrt[G]{g}\subseteq \sqrt[G]{g^n}$, and let $r\in\sqrt[G]{g}$, so that $r^{ plog_G(g)}=g$. Again by taking the $n$-th power, we obtain that 
    $r^{n\cdot plog_G(g)}=g^n$. By assumption $r\in \sqrt[G]{g^n}$, and therefore $r^{ plog_G(g^n)}=g^n $, so that 
\begin{equation*}
r^{ plog_G(g^n)}=r^{n\cdot plog_G(g)}=g^n.
\end{equation*}
As $r$ is a primitive root for $g^n$, it must follow that $plog_G(g^n)= n\cdot plog_G(g)$.
    \end{proof}

    \begin{lemma}
        \label{lemma:common_root}
        Let $G\in\Ups$, and let $x, y \in G$ be such that $x^m = y^n$ for some integers $m,n\geqslant 1$. Then there is $r \in G$ such that $x,y \in \langle r \rangle$. In particular, $\sqrt[G]{x} = \sqrt[G]{y}=\{r\}$.
    \end{lemma}
    \begin{proof}
        Let $r_x = \sqrt[G]{x}$ and $k_x = \plog_G(x)$, so that $r_x^{k_x}=x$, and similarly $r_y = \sqrt[G]{y}$, $k_y = \plog_G(y)$. From primitive stability we obtain that $\plog_G(x^m) = m \cdot\plog_G(x) = m k_x$, and analogously that $\plog_G(y^n) = n k_y$. 
        
        As $x^m = y^n$, we have that $\plog_G(x^m) = \plog_G(y^n)$, that is $m k_x = n k_y$, which we denote by $k$. Therefore $r_x^{k} = r_y^{k}$, and we conclude that $r_x = r_y$, as $G\in \U$.
        Thus $x,y\in\langle r \rangle $, where $r$ denotes  $r_x = r_y$.
    \end{proof}

    \begin{lemma}
        \label{lemma:characterisation_of_p-isolation}
        Let $G\in\Ups$ and $g \in G$ be arbitrary. The subgroup $\langle g \rangle$ is $p$-isolated in $G$ if and only if $\plog_G(g)$ is a power of $p$. Consequently, the subgroup $\langle g \rangle$ is isolated in $G$ if and only if $\plog_G(g) = 1$ or equivalently, if $\langle g \rangle$ is a maximal cyclic subroup of $G$.
    \end{lemma}
    \begin{proof}
        Suppose that $\plog_G(g)$ is not a power of $p$, i.e. $\plog_G(g) = m p^e$ for some $m$ coprime with $p$. For $r = \sqrt[G]{g}$, we have that $r^{p^e} \not\in \langle g \rangle$, but $(r^{p^e})^m = g \in \langle g \rangle$. Hence $\langle g \rangle$ is not $p$-isolated.
        
        Assume now that $\plog_G(g) = p^e$. Let $f \in G$ and suppose that $f^q \in \langle g \rangle$ for some prime $q$ distinct from $p$, so that $f^q = g^k$ for some $k \in \mathbb{Z}$. By Lemma \ref{lemma:common_root} we see that $\sqrt[G]{f} = \sqrt[G]{g} = \{r\}$. 
        
        Set $n = \plog_G(f)$, so that $r^n = f$. We have that 
        \[r^{k p^e} =g^k=f^q= r^{n q},\] 
        and hence $k p^e = n q$. As $q$ is a prime distinct from $p$, it must divide $k$, thus $n = k/q \cdot p^e$ with $k/q$ a natural number. 
        This means that
            \begin{displaymath}
                f = r^n = \left( r^{p^e}\right)^{k/q} = g^{k/q} \in \langle g \rangle,
            \end{displaymath}
        and therefore we proved that $\langle g \rangle$ is $p$-isolated in $G$.
        
        Applying what we just proved, the subgroup $\langle g\rangle$ is separable, that is $p$-separable for all prime numbers $p$, if and only if $\plog_G(g)=1$.
        In view of Lemma \ref{lemma:simplified_p-isolation}, $p$-separability for all primes $p$ is equivalent to the property that if $h^n\in\langle g\rangle$ for some $n\in\mathbb{Z}$, then $h\in\langle g\rangle$ already, and this is equivalent to $\langle g\rangle$ being maximal cyclic in $G$.
    \end{proof}
    Notice that, in the previous lemma, we used the fact that $G$ had the Unique Root property just for one implication.

    Primitive roots and primitive logarithms do not necessarily behave in a stable manner with respect to subgroups. Consider an infinite cyclic group $G = \langle g \rangle$ and let $K = \langle g^k \rangle \leq G$, for some $k \geq 2$. Then $\sqrt[K]{g^{kl}} = g^k$ and $\plog_K(g^{kl}) = l$, whereas $\sqrt[G]{g^{kl}} = g$ and $\plog_G(g^{kl}) = kl$.
    \begin{lemma}
        \label{lemma:primitive_stability_subgroups}
        Let $G\in\Ups$ and $H \leq G$. For every $h \in H$ we have that
            \[\sqrt[H]{h} = (\sqrt[G]{h})^a,\qquad \plog_H(h) = \frac{\plog_G(h)}{a},\]
        where $a = \lvert \langle \sqrt[G]{h}\rangle \colon (\langle \sqrt[G]{h}\rangle \cap H)\rvert$.
    
       Furthermore $H$ is primitively stable, and therefore the class $\Ups$ is closed under taking subgroups.
    \end{lemma}
    \begin{proof}
        Let $r = \sqrt[G]{h}$. By Lemma \ref{lemma:common_root} we see that $a$ divides $\plog_G(h)$, denote $\plog_G(h) = ak$.
        
        Suppose that there is $\tilde r \in H$ such that $\tilde r^{s} = h$ for some $s \in \mathbb{N}$. Using Lemma \ref{lemma:common_root} we see that $\sqrt[G]{\tilde r} =r= \sqrt[G]{r}$, hence $\tilde r \in \langle r\rangle \cap H$. As $a = \min\left\{n \in \mathbb{N} \mid r^n \in H\right\}$, it follows that $\tilde r \in \langle r^a\rangle$. We see that $k = \plog_H(h)$. By unique roots, it follows that $r^a = \sqrt[H]{h}$.
        
        Now consider $h^n$ for some $n \in \mathbb{N}$. By primitive stability in $G$ we see that
        \begin{displaymath}
            \left|\left\langle \sqrt[G]{h^n} \right\rangle \colon \left\langle \sqrt[G]{h^n}\right\rangle \cap H\right| = \left|\left\langle \sqrt[G]{h} \right\rangle \colon \left\langle \sqrt[G]{h}\right\rangle \cap H\right|=a.
        \end{displaymath}
        By the first part of the statement $\plog_H(h^n) = \plog_G(h^n)/a$. By primitive stability in $G$ we see that
        \begin{displaymath}
            \plog_H(h^n) = \frac{\plog_G(h^n)}{a} = n \frac{\plog_G(h)}{a} = n \plog_H(h),
        \end{displaymath}
        hence $H$ is primitively stable.
    \end{proof}

    However, primitive roots and primitive logarithms are stable with respect to retracts.
    \begin{lemma}
    \label{lemma:primitive_roots_retract}
        Let $G\in\U$, suppose that $R \leq G$ is a retract such that $R$ is primitively stable, and let $\rho \colon G \to R$ denote the corresponding retraction. Then $\plog_G(g)$ is defined for all $g \in G$ with $\rho(g) \neq e$ and it divides $\plog_R(\rho(g))$. Moreover, if $g \in R$, then $\sqrt[G]{g} = \sqrt[R]{g}$ and $\plog_R(g) = \plog_G(g)$.  
    \end{lemma}
    \begin{proof}
        Suppose that $g = r^k$ for some $r \in G$ and $k \in \mathbb{N}$. Then $\rho(r)^k = \rho(g)$. Let $\tilde r = \sqrt[R]{\rho(g)}$ and $\tilde k = \plog_R(\rho(g))$. We see that 
        \[{\tilde r}^{\tilde k} = \rho(g) = \rho(r)^k,\] 
        and therefore by primitive stability we obtain that
        \[\plog_R(\rho(g))=\plog_R(\rho(r)^k)=k\cdot\plog_R(\rho(r)).\]
        It follows that $k$ divides $\plog_R(\rho(g))$, hence $\plog_G(g)$ is defined and we see that it has to divide $\plog_R(\rho(g))$.
        
        Suppose that $g\in R$, so that $\rho(g) = g$. Then $\rho(r)^k = \rho(g) = g = r^k$ and by the Unique Root property we see that $\rho(r) = r$, i.e. $\sqrt[G]{g} \in R$. By Lemma~\ref{lemma:common_root} we see that $ r=\sqrt[R]{g}$ and $k = \plog_R(g)$.
    \end{proof}
    Let us note here that the assumption on non-triviality of $\rho(g)$ is essential. One can check that the Bauslag-Solitar group 
    \begin{displaymath}
       \BS(1,p) = \langle a, t \parallel tat^{-1} = a^p \rangle = (\mathbb{Z}[1/p], +) \rtimes \langle t \rangle
    \end{displaymath} 
    has Unique Root property and the subgroup $\langle t \rangle$ is a retract with canonical retraction $\rho \colon \BS(1,p) \to \langle t \rangle$. Nevertheless $\rho(a) = e$, and $\plog_{\BS(1, p)}(a)$ is not defined.
    
\begin{definition}[{\bf $p$-inseparability}]    
Let $G$ be a group and $f,g \in G$. We say that the pair $(f, \langle g \rangle)$ is $p$-\emph{inseparable} if $\pi_N(f) \in \langle \pi_N(g) \rangle$ for every $N \in \Np(G)$, where $\pi_N\colon G\to G/N$ is the canonical projection.
\end{definition}
    
    The following lemma shows that for groups in $\Ups$ which are $p$-CSS, the only inseparable pairs arise from the counterexamples to $p$-isolation (see Example \ref{example:p_inseparable_pair}).
    \begin{lemma}
    	\label{lemma:p_inseparable_pair}
        Let $G\in\Ups$ be a $p$-CSS group and let $f,g \in G$ be such that $f \notin \langle g\rangle$ and the pair $(f, \langle g \rangle)$ is $p$-inseparable. Then $\sqrt[G]{f} = \sqrt[G]{g}$ and $f^k \in \langle g \rangle$ for some $k \in \mathbb{Z}$ coprime to $p$. 
    \end{lemma}
    \begin{proof}
        As $G$ is $p$-CSS, we see that $g$ cannot be $p$-isolated. Set $r = \sqrt[G]{g}$. As $g$ is not $p$-isolated, from Lemma \ref{lemma:characterisation_of_p-isolation} we see that $\plog_G(g)= k p^{e}$, where $p$ does not divide $k>1$. 
        
        Set $g' = r^{p^{e}}$. Notice that $\plog_G(g') = p^{e}$, hence $g'$ is $p$-isolated in $G$ by Lemma \ref{lemma:characterisation_of_p-isolation} and, in particular, $\langle g' \rangle$ is $p$-closed in $G$. 
        
        If $f \not\in \langle g' \rangle$, then there is $N \in \Np$ such that $\pi_N(f) \not \in \langle \pi_N(g') \rangle$. As $\langle g \rangle \leq \langle g' \rangle$, it would follow that $\pi_N(f) \not\in \langle \pi_N(g) \rangle$, contradicting the fact that $(f, \langle g\rangle)$ is $p$-inseparable. 
        
        Therefore $f \in \langle g' \rangle$ and $f^k \in \langle g \rangle$. Moreover, by Lemma \ref{lemma:common_root} we conclude that  $\sqrt[G]{f} = \sqrt[G]{g}$.
    \end{proof}
    
\section{Primitive stability in graph products}
\label{section:primitive_stabilility_closure}
The following lemma is a direct consequence of Lemma \ref{lemma:primitive_roots_retract} together with Definition \ref{definition:plogs} and Definition \ref{definition:primitive_stability}.
\begin{lemma}
    \label{lemma:plog_direct}
        Let $G_1, G_2\in \Ups$ and $g = (g_1, g_2) \in G_1 \times G_2 = G$ be arbitrary. Then
        \begin{displaymath}
            \plog_{G}(g) = \gcd\left(\plog_{G_1}(g_1), \plog_{G_2}(g_2)\right).
        \end{displaymath}
        and
        \begin{equation}\label{eq:plog_direct2}
            \sqrt[G]{g} = \left(\sqrt[G_1]{g_1}^{\frac{\plog_{G_1}(g_1)}{\plog_{G}(g)}}, \sqrt[G_2]{g_2}^{\frac{\plog_{G_2}(g_2)}{\plog_{G}(g)}}\right).
        \end{equation}
    \end{lemma}

    \begin{proof}
The groups $G_1$ and $G_2$ are both retracts of the direct product $G_1\times G_2$. Therefore, from Lemma \ref{lemma:primitive_roots_retract} it follows that 
$\plog_{G}(g)$ divides both $\plog_{G_1}(g_1)$ and $\plog_{G_2}(g_2)$, and hence their greatest common divisor. As the primitive logarithm is defined to be the exponent of the primitive root, it must be that $\plog_{G}(g)$ is that greatest common divisor.

The equality of equation \eqref{eq:plog_direct2} follows from the previous argument.

    \end{proof}

For the next fact, let us remember that if $G$ is a free abelian group freely generated by $a_1, \dots, a_n$, then any element $g \in G$ can be written as $g = (a_1^{k_1}, \dots, a_n^{k_n})$ for uniquely given $k_1, \dots, k_n \in \mathbb{Z}$.


    \begin{remark}
    Notice that primitive stability is an essential hypothesis for Lemma \ref{lemma:plog_direct}. Indeed, consider the groups
    \[G_n:=\langle g_n,r_{n,2},\dots, r_{n,n}\parallel g_n=r_{n,i}^i\ \forall i=2,\dots,n\rangle.\]
    for $n\geqslant 2$.
    These groups are not primitively stable, and $G_n\times G_{n-1}$ does not satisfy the conclusion of Lemma \ref{lemma:plog_direct}, because
    \[\plog_{G_n\times G_{n-1}}(g_n,g_{n-1})=n-1\]
    does not divide the great common divisor of $n-1$ and $n$.
    \end{remark}
    
\begin{lemma}
    \label{lemma:primitive_urp_direct}
	The class $\Ups$ is closed under direct products.
\end{lemma}
\begin{proof}
	Let $G_1, G_2\in\Ups$, consider $G = G_1 \times G_2$ and let $n \in \mathbb{N}$, $g = (g_1, g_2) \in G$ be arbitrary. By Lemma \ref{lemma:urp_direct} we see that $G$ has Unique Root property, that is $G\in\U$. From Lemma \ref{lemma:plog_direct}, and exploiting primitive stability for the second equality, we get that
	\begin{align*}
	    \plog_G(g^n)    &= \gcd \left(\plog_{G_1}(g_1^n), \plog_{G_2}(g_2^n)\right)\\
	                    &= \gcd \left(n \plog_{G_1}(g_1), n\plog_{G_2}(g_2)\right)\\
	                    &=n \gcd \left(\plog_{G_1}(g_1), \plog_{G_2}(g_2)\right) = n \plog_G(g).
	\end{align*}
	Therefore $G$ is also primitively stable.
\end{proof}

\begin{lemma}
    \label{lemma:amalgam_common_root}
    Let $G_1 = K_1\rtimes R$, $G_2 = K_2 \rtimes R$ be groups in $\U$ and suppose that $R$ is primitively stable. Let $x,y \in G = G_1 \ast_R G_2$ be cyclically reduced elements with $\lvert x\rvert_*, \lvert y\rvert_* > 1$, and suppose that $x^m = y^n$ for some $m,n \in \mathbb{N}$. Then there are $z \in G$ and $a, b \in \mathbb{N}$ such that $z^a= x$ and $z^b = y$.
\end{lemma}
\begin{proof}
    Let $\rho \colon G \to R$ be the canonical retraction. As $\rho(x)^m = \rho(y)^n$, by Lemma \ref{lemma:common_root} we see that $\rho(x)$ and $\rho(y)$ have a common primitive root in $R$, denoted by $r_0$. Following Proposition \ref{lemma:urp_amalgams} we see that $G\in\U$, therefore by Lemma \ref{lemma:primitive_roots_retract}
    \[\plog_R(\rho(x)) = \plog_G(\rho(x)),\qquad \sqrt[R]{\rho(x)} = \sqrt[G]{\rho(x)},\] 
    and the analogous equalities hold also for $\rho(y)$. Set 
    \begin{displaymath}
        c = \gcd\bigl(\plog_R(\rho(x)), \plog_R(\rho(y))\bigr)
    \end{displaymath}
    and
    \begin{displaymath}
        a = \frac{\plog_R(\rho(x))}{c}, \qquad b = \frac{\plog_R(\rho(y))}{c},
    \end{displaymath}
    so that $\gcd(a,b) = 1$. Denoting $r := r_0^c$, we have that $r_0^{ac}=r^a = \rho(x)$, and $r^b = \rho(y)$, and therefore $m a = n b$.
    
    As both $x,y$ are cyclically reduced, and $|x|_*, |y|_* > 1$, we see that $|x^m|_* = m|x|_*$ and $|y^n|_* = n|y|_*$ by Lemma \ref{lemma:amalgam_CR_powers}. Thus $m|x|_* = n|y|_*$, and 
    in fact
    \[|x|_* = a l,\qquad |y|_* = b l,\] 
    where $l = \gcd(|x|_*, |y|_*)$. Notice that $2\mid l$, and therefore $l\neq 1$.
    
    Let $x=k_1\dots k_{al}r^a$ be the reduced expression for $x$. By Corollary \ref{cor:spelt_out}, we have that
    \begin{equation}\label{eq_mad}
    y^n=x^m = k_1 \dots k_{al} k_1^{r^a} \dots k_{al}^{r^a}\dots k_1^{r^{a(m-1)}} \dots k_{al}^{r^{a(m-1)}}r^{am}
    \end{equation}
    is the (unique) reduced expression for $x^m = y^n$, where $k_1,\dots, k_{al}\in\ K_1\cup K_2$. Progressively rename the $aml$ elements of $K_1\cup K_2$ in the right-hand side of equation \eqref{eq_mad} as $k_1,\dots, k_{aml}$, so that $x^m=k_1\dots k_{aml}r^{am}$.
    
    By assumption we have $k_{i + al} = k_i^{r^a}$ for $i = 1, \dots, (m-1)al$, as is deduced from equation \eqref{eq_mad}. Moreover, we claim that $k_{i+bl} = k_i^{r^b}$ for $i = 1, \dots, (n-1)bl$ too.
    Indeed, let $\tilde k_1,\dots, \tilde k_{bl}\in K_1\cup K_2$ be 
    elements such that $y=\tilde k_1\dots \tilde k_{bl}r^b$ is the reduced expression for $y$. Then, by Corollary \ref{cor:spelt_out}, we have that
    \begin{equation}\label{eq_mad2}
        x^m=y^n = \tilde k_1 \dots \tilde k_{bl} \tilde k_1^{r^b} \dots \tilde k_{bl}^{r^b}\dots \tilde k_1^{r^{b(n-1)}} \dots \tilde k_{bl}^{r^{b(n-1)}}r^{bn} 
    \end{equation}
    is a reduced expression for $x^m=y^n$. As the reduced expression is unique, it must be that $k_j=\tilde k_j$ for all $j=1,\dots, bl$, where the $\tilde k_j$ are the elements appearing in equation \eqref{eq_mad2} and the $k_j$ are the (renaming of the) elements appearing in equation \eqref{eq_mad}.
    By construction, as seen in equation \eqref{eq_mad2}, we have that $\tilde k_{i+bl} = \tilde k_i^{r^b}$ for $i = 1, \dots, (n-1)bl$, and therefore, as claimed,
    $k_{i+bl} = k_i^{r^b}$ for $i = 1, \dots, (n-1)bl$.

    
    Suppose that $c_1, c_2 \in \mathbb{Z}$ are B\'ezout's coefficients for $\plog_R(\rho(x))$ and $\plog_R(\rho(y))$, i.e.
    \begin{displaymath}
        c = \gcd\bigl(\plog_R(\rho(x)), \plog_R(\rho(y))\bigr) = c_1 \plog_R(\rho(x)) + c_2 \plog_R(\rho(y)).
    \end{displaymath}
    As both $\plog_R(\rho(x))$ and $\plog_R(\rho(y))$ are positive integers, it must be that $c_1c_2<0$. Without loss of generality, suppose that $c_1<0$.
    It then follows that $c_1, c_2$ are also B\'ezout's coefficients for $a,b$, that is $1=c_1 a + c_2 b$. In particular $c_1 al + c_2 bl = l$, and consequently
    \begin{displaymath}
    \begin{split}
        k_{i+l} &= k_{i + c_1 al + c_2 bl} = \bigl(k_{i + c_1 al + (c_2-1) bl}\bigr)^{r^b}=\dots\\
        &=k_i^{r^{c_1 a + c_2b}} = k_i^r
    \end{split}
    \end{displaymath}
    for all $i=1,\dots,(am-1)l$.

    If we set $z = k_1 \dots k_l r \in G$, we see that $z^a = x$ and $z^b = y$.
\end{proof}

\begin{proposition}
    \label{lemma:primitive_urp_amalgam}
    The class $\Ups$ is closed under taking amalgams over retracts.
\end{proposition}
\begin{proof}
    Let $G_1 = K_1 \rtimes R$, $G_2 = K_2 \rtimes R$ be groups in $\Ups$, and consider $G = G_1 \ast_R G_2 \simeq (K_1 \ast K_2)\rtimes R$, their amalgam along the common retract~$R$. 
    
    By Proposition \ref{lemma:urp_amalgams} we have that $G\in\U$. To prove that $G$ is also primitively stable, consider an element $g \in G$ and $n \in \mathbb{N}$. Following Remark \ref{remark:plog_conjugacy}, without loss of generality we may assume that $g$ is cyclically reduced. There are two distinct cases to consider: either $g$ belongs to one of the factors, or not.
    
    \smallskip
    Suppose that $g$ belongs to one of the factors, assume $g \in G_1$. The group $G_1$ is a retract of $G$, with the retraction map $\rho_1$ defined on the generators of $G$ in the following manner:
    \begin{displaymath}
        \rho_1(g) = \begin{cases}
            g & \mbox{ if $g \in R$},\\
            g & \mbox{ if $g \in K_1$},\\
            1 & \mbox{ if $g \in K_2$}.
        \end{cases}
    \end{displaymath}
    From Lemma \ref{lemma:primitive_roots_retract} we see that $\plog_G(g)$ is defined. As $G \in \U$, we see that $\sqrt[G]{g}=\sqrt[G_1]{g}$. Moreover
    \begin{equation*}\label{equation:primitive_urp_amalgam1}
        \plog_G(g^n) = \plog_{G_1}(g^n) = n\plog_{G_1}(g) = n\plog_G(g),
    \end{equation*}
    therefore the element $g$ is primitely stable.
    
    \smallskip
    For the remaining case, suppose that $g$ does not belong to a factor. Therefore $\lvert g\rvert_*>1$. Let
    \[x = \sqrt[G]{g},\qquad m = \plog_G(g)\] 
    and 
    \[y = \sqrt[G]{g^n},\qquad e = \plog_G(g^n).\]
    Note that  $\plog_G(g)$ and $\plog_G(g^n)$ are defined by Lemma \ref{lemma:amalgam_CR_powers} and thus $x,y$ are given uniquely as $G \in \U$.
    
    As $g$ is cyclically reduced, by Lemma \ref{lemma:amalgam_CR_powers} we see that the three elements $x$, $x^{mn} = g^n = y^e$, and $y$ are cyclically reduced. 
    By definition of primitive logarithm we have that 
    \[e=\plog_G(g^n)\geqslant n\plog_{G}(g) =mn.\] 
    By Lemma \ref{lemma:amalgam_common_root}, there is an element $z \in G$ and natural numbers $a,b \in \mathbb{N}$ such that $z^a = x$ and $z^b = y$. Therefore $g^n = z^{be}$, which is a contradiction with the maximality of $\plog_G(g^n)$ unless $b = 1$. Hence $z = y$, by unique roots. Similarly $z^a = x$ and, consequently, $g = z^{am}$. Again, this implies that $a = 1$, and therefore $x = y$, that is $\sqrt[G]{g^n} = \sqrt[G]{g}$ and $\plog_G(g^n) = n\plog_G(g)$.
\end{proof}

In view of Lemma \ref{lemma:primitive_urp_direct} and Proposition \ref{lemma:primitive_urp_amalgam}, from Theorem \ref{general:GP_CSS} we deduce:
\begin{theorem}
    \label{lemma:primitive_urp_gp}
	The class $\Ups$ is closed under taking graph products.
\end{theorem}
As an immediate corollary:
\begin{corollary}
    \label{cor:raag_ps}
    Right-angled Artin groups belong to $\Ups$.
\end{corollary}

\section{Examples of primitively stable groups with Unique Root property}
\label{section:examples_of_Ups}
Let $G$ be a group and $g \in G$. We define \emph{the radical of $g$ in $G$} as the subset
\begin{displaymath}
    \Rad_G(g) := \bigl\{r \in G \mid r^a \in \langle g \rangle \mbox{ for some } a \in \mathbb{Z}\setminus\{0\}\bigr\}.
\end{displaymath}
Notice that the set $\Rad_G(g)$ does not have to be a subgroup of $G$. Indeed, let $G=\langle x, y \parallel x^2 = y^2 \rangle$ be the fundamental group of the Klein bottle. Then $x,y \in \Rad_G(x^2)$, but $xy \notin \Rad_G(x^2)$. 

Nevertheless, in Proposition \ref{lemma:cyclic_radical_criterion} we prove that $\Rad_G(g)$ is a cyclic subgroup for any non-trivial element $g$ of a group $G\in\Ups$.

In the context of nilpotent groups, our notion of the radical of an element $g$ coincides with the notion of isolator of the subgroup $\langle g\rangle$. In particular, if $G$ is a nilpotent group, then for every element $g \in G$ the radical $\Rad_G(g)$ is a subgroup of $G$ (see \cite[Theorem 2.5.8]{khukhro_nilpotent}).

\begin{lemma}
\label{lemma:radical_all_or_nothing}
Let $G$ be a group and let $g_1, g_2 \in G \setminus \{e\}$ be arbitrary. If $\Rad_G(g_1) \cap \Rad_G(g_2) \neq \{e\}$ then $\Rad_G(g_1) = \Rad_G(g_2)$.
\end{lemma}
\begin{proof}
    Suppose that there exists $r \in \Rad_G(g_1) \cap \Rad_G(g_2)$ such that $r\neq e$. This means that there are $a, b, m, n\in  \mathbb{Z}\setminus \{0\}$ such that $g_1^a = r^m$ and $g_2^b = r^n$. Let $s \in \Rad_G(g_1)$ be arbitrary and let $c, k \in \mathbb{Z}\setminus\{0\}$ be such that $g_1^{k} = s^{c}$. It follows that
    \begin{displaymath}
        (s^{c})^{an} = (g_1^{k})^{an} = (g_1^a)^{k n} = (r^m)^{k n} = (r^n)^{k m} = (g_2^b)^{k n}
    \end{displaymath}
    and we see that $s \in \Rad_G(g_2)$. Consequently $\Rad_G(g_1) \subseteq \Rad_G(g_2)$. The opposite inclusion can be shown analogously, and therefore $\Rad_G(g_1)= \Rad_G(g_2)$.
\end{proof}

\begin{proposition}
    \label{lemma:cyclic_radical_criterion}
   Let $G$ be a torsion-free group. An element $g\in G$ is primitively stable with unique roots in $G$ if and only if $\Rad_G(g)$ is an infinite cyclic subgroup of $G$.
\end{proposition}
\begin{proof}
    If $g$ is primitively stable with unique roots, then exploiting Lemma \ref{lemma:common_root} one can show that the element $\sqrt[G]{g}$ generates $\Rad_G(g)$.
    
    Assume that $\Rad_G(g)$ is cyclic, and let $r \in \Rad_G(g)$ be a generator such that $g = r^n$ for some $n \in \mathbb{N}\setminus\{0\}$.
    We now show that $r=\sqrt[G]{g}$, that $n=\plog_G(g)$, and that $g$ is primitively stable, in this order.
    
    For this purpose, let $h_1, h_2 \in G$ and suppose that $h_1^m = g = h_2^m$ for some $m \in \mathbb{N}$. By definition $h_1, h_2 \in \Rad_G(g)$, hence $h_1 = h_2$ because the infinite cyclic group has unique roots. This means that $g$ has unique roots in $G$ and $r = \sqrt[G]{g}$.
    
    Assume now that $h \in G$ is such that $h^a = g$ for some $a \in \mathbb{N}$. By definition $h \in \Rad_G(g)$, and therefore $h = r^b$ for some $b \in \mathbb{Z}$. We see that 
    \[ r^n=g=h^a=r^{ab}.\] 
    The element $r$ has infinite order, so $ab = n$, hence $b > 0$ and $a \leq n$. It follows that $n$ is maximal, that is $n = \plog_G(g)$.
    
    Finally, we show that $\plog_G(g^m)=m\cdot\plog_G(g)$ for all $m\in\mathbb{N}$. Suppose that there is $h \in G$ such that $h^k = g^m$ for some $k\in\mathbb{N}$. Again, $h \in \Rad_G(g)$ so $h = r^a$ for some $a \in \mathbb{Z}$. As $g=r^n$, we see that $r^{nm} = r^{ak}$, hence $nm = ak$. It follows that $a > 0$ and that $m$ divides $kn$
    . We see that $\plog_G(g^m) = nm = m \plog_G(g)$, and therefore $g$ is primitively stable.
\end{proof}
In the light of this proposition, we say that $g$ has \emph{cyclic radical} in $G$ whenever it is primitively stable and with unique root.

Note that Theorem \ref{thm:C} is an immediate consequence of Proposition \ref{lemma:cyclic_radical_criterion}

\subsection{Residually torsion-free nilpotent groups}
In this subsection we prove that residually finitely generated torsion-free nilpotent groups belong to $\Ups$, that is they are primitively stable and with unique roots.
    \begin{lemma}
        \label{lemma:tf_radical_embedding}
        Let $G$ be a torsion-free group let $g \in G\setminus  \{e\}$ be arbitrary. Suppose that there is $N \unlhd G$ such that $g \notin N$ and $G/N$ is torsion-free. Then  $\Rad_G(g) \cap N = \{e\}$; in particular, if $\Rad_G(g)$ is a subgroup of $G$ then the canonical projection $\pi \colon G \to G/N$ is injective on $\Rad_G(g)$.
    \end{lemma}
    \begin{proof}
        Let $r \in \Rad_G(g)$ be nontrivial, i.e. $r^a = g^b$ for some $a,b \in \mathbb{Z} \setminus \{0\}$. As $g \notin N$ and $G/N$ is torsion-free, we see that $\pi(r)^a = \pi(g)^b \neq e$. It follows that $\pi(r) \neq e$.
    \end{proof}
    
    Before proceeding to the proof of the following proposition, let us recall that in the case of nilpotent groups the radical is always a subgroup by \cite[Theorem 2.5.8]{khukhro_nilpotent}.
    \begin{proposition}
        \label{lemma:fgtfn_ups}
       If $G$ is a finitely generated torsion-free nilpotent group then $G \in \Ups$.
    \end{proposition}
    \begin{proof}
        We will proceed by induction on the nilpotency class. If $G$ is 1-step nilpotent, then $G$ is a torsion-free abelian group. As $G$ is finitely generated, we see that $G$ is in fact free abelian. Clearly, free abelian groups belong to the class $\Ups$.
        
        Now suppose the statement holds for all finitely generated torsion-free groups of nilpotency class $n-1$, and suppose that $G$ is $n$-step nilpotent. Let $\{e\} = Z_0 \leq \dots \leq Z_n = G$ denote the upper central series of $G$ and let $\pi_i \colon G \to G/Z_i$ denote the corresponding canonical projections. Recall that in the case of torsion-free nilpotent groups, the quotient $G/Z_i$ is torsion-free for every $i \in \{0, \dots, n-1\}$. Let $g \in G$ be non-trivial and let $r \in \Rad_G(g)$ be nontrivial as well, repeating the argument from Lemma \ref{lemma:tf_radical_embedding} we see that $\pi_i(r)$ is trivial in $G/Z_i$ if and only if $\pi_i(g)$ is. Pick $i \in \{0, \dots, n-1\}$ such that $g \in Z_{i+1} \setminus Z_i$. We see that $\Rad_G(g)\setminus\{e\} \subseteq Z_{i+1} \setminus Z_i$.
        
        Suppose that $i = 0$, i.e. $\Rad_G(g)$ is contained in $Z_1$, the center of $G$. As finitely generated nilpotent groups are slender, i.e. every subgroup is finitely generated, we see that the center $Z_1$ is finitely generated and therefore a free abelian group. We see that $\Rad_G(g) = \Rad_{Z_1}(g)$ and hence it must be cyclic.
        
        If $i > 0$, then the group $G/Z_i$ is a finitely generated group of nilpotency class $n-i$. By induction hypothesis we see that $\Rad_{G/Z_i}(\pi_i(g))$ is cyclic. Clearly, $\pi_i(\Rad_G(g)) \subseteq \Rad_{G/Z_i}(\pi_i(g))$. As $\pi_i$ is injective on $\Rad_G(g)$ we see that $g$ has a cyclic radical in $G$.
    \end{proof}
    Note that finite generation is essential in Proposition \ref{lemma:fgtfn_ups}. Indeed, consider the group given by the presentation 
        \begin{displaymath}
            G_p=\langle a_0, a_1, \dots \parallel a_{i+1}^p = a_i \mbox{ for }i=0,1,\dots\rangle\cong \mathbb{Z}[p^{-1}].
        \end{displaymath}
    Obviously, $G_p$ is a torsion-free abelian group, hence of nilpotency class one, and it is not finitely generated. It can be seen that $\Rad_{G_p}(a_0) = G_p$.

    \smallskip
    Combining Lemma \ref{lemma:tf_radical_embedding} and Proposition \ref{lemma:fgtfn_ups} we immediately obtain the following result.
     \begin{corollary}
        If $G$ is a residually finitely generated torsion-free nilpotent group such that $\Rad_G(g)$ is a subgroup of $G$ for every $g \in G$, then $G \in \Ups$.
    \end{corollary}

\subsection{Hyperbolic and relatively hyperbolic groups}
Recall that a group is called \emph{elementary} if it contains a cyclic subgroup of finite index. The following was proved by Ol'shanskii \cite[Lemma 1.16]{olshanskii}.

\begin{lemma}
    \label{lemma:hyp_elementary}
   An infinite-order element $g$ of a hyperbolic group $G$ is contained in a unique maximal elementary subgroup, denoted by $\E_G(g)$.
   
   An element $x \in G$ belongs to $E_G(g)$ if and only if there exists $n \in \mathbb{Z}\setminus \{0\}$ such that $xg^nx^{-1} = g^{\pm n}$.
\end{lemma}
From this we can deduce that torsion-free hyperbolic groups belong to~$\Ups$.
\begin{lemma}
   If $G$ is a torsion-free hyperbolic group then $G \in \Ups$.
\end{lemma}
\begin{proof}
    Let $g \in G$ be non-trivial, and notice that $\Rad_G(g) \leq E_G(g)$. Indeed, let $r\in\Rad_G(g)$, so that there exist $a,b\in\mathbb{Z}\setminus\{0\}$ for which $r^a=g^b$. We have that
    \[g^b=r^a=rg^br^{-1},\]
    and by Lemma \ref{lemma:hyp_elementary} we conclude that $r\in E_G(g)$.
    As $G$ is torsion-free, we see that $E_G(g)$ is cyclic, as a torsion-free elementary group is necessarily cyclic. Since $E_G(g)$ is infinite cyclic, $\Rad_G(g)=E_G(g)$.
\end{proof}

Let $G$ be a group and let $\mathcal{H} = \{H_i \mid i \in I\}$ be a collection of subgroups of $G$, where $I$ is a set. Then $G$ is \emph{hyperbolic relative to} $\mathcal{H}$ if the coned-off Cayley graph (see \cite[Definition 6.3]{rel_hyp_osin} for the definition of coned-off Cayley graph)
is hyperbolic and fine, in the sense of Bowditch \cite{bowditch}. We refer to \cite{rel_hyp_osin} for more on relatively hyperbolic groups and equivalent definitions.

The following lemma is an easy corollary of \cite[Theorem 1.4, Theorem~1.5]{rel_hyp_osin}.
\begin{lemma}
    \label{lemma:parabolic_intersections}
   Let $G$ be torsion-free group and let $\mathcal{H} = \{H_i \mid i\in I\}$ be a collection of subgroups of $G$. If $G$ is hyperbolic relative to $\mathcal{H}$ then the following are true:
    \begin{itemize}
        \item[(i)] for any $g_1, g_2 \in G$ the intersection $H_i^{g_1} \cap H_j^{g_2}$ is trivial whenever $i \neq j$;
        \item[(ii)] the intersection $H_i^g \cap H_i$ is trivial whenever $g \not\in H_i$.
    \end{itemize}
\end{lemma}

\begin{lemma}
    \label{lemma:parabolic_radical}
   Let $G$ be torsion-free group and suppose that $G$ is hyperbolic relative to a collection of subgroups $\{H_i \mid i\in I\}$. If $g \in H_i$, then $\Rad_G(g) \subseteq H_i$, i.e. $\Rad_{H_i}(g) = \Rad_G(g)$.
\end{lemma}
\begin{proof}
    Let $g$ be as above and let $r \in \Rad_G(g)$, i.e. $r^a = g^b$ for some $a,b \in \mathbb{Z}\setminus \{0\}$. Clearly, $r g^b r^{-1} = r r^a r^{-1} = r^a = g^b$, so $\langle g^b \rangle \leq H_i \cap H_i^r$. Following Lemma \ref{lemma:parabolic_intersections}, this is a contradiction unless $r \in H_i$.
\end{proof}
An element is said to be hyperbolic if its conjugacy class does not intersect any of the subgroups in $\mathcal{H}$. The following generalisation of Lemma \ref{lemma:hyp_elementary} was proved in \cite[Theorem 4.3]{osin_elementary}.
\begin{lemma}
    \label{lemma:rel_hyp_elementary}Let $G$ be hyperbolic relative to the family $\mathcal{H}$.
   Every hyperbolic element $g \in G$ of infinite order is contained in a unique maximal elementary subgroup $E_G(g)$, and moreover
   \begin{displaymath}
       E_G(g) = \bigl\{x \in G \mid x g^n x^{-1} = g^{\pm n} \mbox{ for some }n \in \mathbb{Z} \setminus\{0\}\bigr\}.
   \end{displaymath}
\end{lemma}
From this we deduce:

\begin{proposition}
   Let $G$ be a torsion-free group and suppose that $G$ is hyperbolic relative to the collection of subgroups $\{H_i \mid i\in I\}$. Then $G \in \Ups$ if and only if $H_i \in \Ups$ for all $i \in I$.
\end{proposition}
\begin{proof}
    By Lemma \ref{lemma:primitive_stability_subgroups} the class $\Ups$ is closed under taking subgroups. Suppose therefore that $H_i \in \Ups$ for all $i \in I$,
    and let $g \in G$ be a non-trivial element. If $g$ is hyperbolic then, following Lemma \ref{lemma:rel_hyp_elementary}, $E_G(g)$ is cyclic and $\Rad_G(g) = E_G(g)$. Hence $g$ has cyclic radical.
    
    Suppose that $g \in H_i$. By Lemma \ref{lemma:parabolic_radical} we see that $\Rad_G(g) = \Rad_{H_i}(g)$ which is cyclic by assumption. Finally, if $ygy^{-1}\in H_i$ for some $y\in G$, then 
    $\Rad_G(g)=y\Rad_G(ygy^{-1})y^{-1}$, and therefore also in this case $g$ has cyclic radical.
\end{proof}

\begin{corollary}
    If $G$ is a toral relatively hyperbolic group then $G \in \Ups$. In particular, if $G$ is a limit group then $G \in \Ups$.
\end{corollary}

\section{Separating cyclic subgroups of graph products in the pro-$p$ topology}
\label{section:prop_case}
We start the section by noting that finitely generated free abelian groups are $p$-CSS. In fact, we provide a short proof that any $p$-isolated subgroup of a free abelian group is $p$-closed:
    \begin{proposition}\label{lemma:free_abelian_p-CSS}
        Let $G$ be a free abelian group. Then every finitely generated $p$-isolated subgroup of $G$ is separable in $\prop(G)$.
        
    \end{proposition}
    \begin{proof}

Without loss of generality, we can assume that $G$ is a finitely generated free abelian group, and consider a $p$-isolated subgroup $C\leqslant G$.
Then, the only torsion elements in the quotient $G/C$ are $p$-torsion elements. Therefore, $G/C$ is a residually $p$-finite group, by the classification of finitely generated abelian groups. Therefore,
since the quotient is $G/C$ residually $p$-finite, the subgroup $C$ is $p$-separable in $G$.
    \end{proof}
    
On the other hand, notice that the group $\mathbb{Q}$
 of rational numbers is not $p$-CSS, for any prime number $p$. Indeed, the trivial subgroup, which is the only $p$-separable subgroup of $\mathbb Q$ being the group torsion-free and divisible, is not $p$-closed for any $p$, because this would imply that $\mathbb{Q}$ is a residually $p$-finite group.
    
    \smallskip
    To see that finite generation in Proposition \ref{lemma:free_abelian_p-CSS} is essential, consider now a free abelian group $G=\prod_{i\in \mathbb N}\mathbb Z$ with countable basis, and let $K_p=\ker \pi$, where $\pi\colon G\to \mathbb Z_{p^\infty}$ is the canonical projection onto the $p$-Pr\" ufer group $Z_{p^\infty}=\langle x_0,x_1,\dots\parallel x_0^p=e,x_1^p=x_0,\dots, x_n^p=x_{n-1,\dots}\rangle$ defined by $\pi(z_i)=x_i$, where $z_i$ is the generator of the $i$-th copy of $\mathbb Z$ in $G$. Then, $K_p$ 
    is not finitely generated and it is $p$-isolated. Moreover, as the quotient $G/K_p$ is not residually $p$-finite, $K_p$ cannot be $p$-closed in $G$.

 \begin{lemma}
        \label{lemma:restriction_p_cyclic}
        Let $G\in \Ups$ be a $p$-CSS group, and suppose that $C \leq G$ is cyclic and $p$-isolated. Then $\prop(C)$ is the restriction of $\prop(G)$.
    \end{lemma}
    \begin{proof}
        Let $g \in G$ be a generator of $C$ and set $r = \sqrt[G]{g}$. Following Lemma \ref{lemma:characterisation_of_p-isolation} we see that $\plog_G(g)$ is a power of $p$, i.e. $g = r^{p^e}$ for some $e \in \mathbb{N}$. As $G$ is $p$-CSS, $C$ is $p$-closed in $G$. Now let $N \in \Np(C)$ be arbitrary. As $C$ is infinite cyclic, the subgroup $N$ is cyclic as well, let $f \in G$ be the generator. Clearly, $f = g^{p^k}$ for some $k \in \mathbb{N}$. We see that $f = g^{p^k} = (r^{p^e})^{p^k} = r^{p^{e+k}}$. By primitive stability $\plog_G(f) = p^{e+k}$, hence $N$ is $p$-isolated in $G$ by Lemma \ref{lemma:characterisation_of_p-isolation}.
        Thus it is $p$-separable in $G$, being $G$ a $p$-CSS group. The statement follows using Lemma \ref{lemma:restriction}.
    \end{proof}
    
    \begin{remark}
    	\label{remark:isolation_hereditary}
    	Let $G$ be a group and let $H_1, H_2$ be two subgroups such that $H_1 \leq H_2$. If $H_1$ is $p$-isolated in $G$ then $H_1$ is $p$-isolated in $H_2$.
    \end{remark}

    \begin{lemma}
        \label{lemma:direct_pCSS}
        The class of primitively stable $p$-CSS groups with Unique Root property is closed under forming direct products.
    \end{lemma}
    \begin{proof}
        Let $G_1, G_2\in \Ups$ be $p$-CSS groups, consider $G = G_1 \times G_2$, and note that $G\in \Ups$ by Lemma \ref{lemma:primitive_urp_direct}. Let $g = (g_1, g_2) \in G$ be a $p$-isolated element. Set $r_i = \sqrt[G_i]{g_i}$ and denote $R_i = \langle r_i \rangle \leq G_i$, for $i=1,2$.
        
        Note that $\plog_{G_1}(r_1) = 1 =\plog_{G_2}(r_2)$, hence $R_1$ is $p$-isolated in $G_1$ and $R_2$ is $p$-isolated in $G_2$ in view of Lemma \ref{lemma:characterisation_of_p-isolation}. As both $G_1$ and $G_2$ are $p$-CSS, we see that $R_1$ is $p$-closed in $G_1$ and $R_2$ is $p$-closed in $G_2$. 
        
        Using Lemma \ref{lemma:restriction_p_cyclic} we see that $\prop(R_1)$ is a restriction of $\prop(G_1)$ and, similarly, $\prop(R_2)$ is a restriction of $\prop(G_2)$. It follows by Lemma \ref{lemma:restriction_product} that $\prop(R_1 \times R_2)$ is a restriction of $\prop(G)$. As $R_1 \times R_2$ is a free abelian group on two generators, all its $p$-isolated cyclic subgroups are $p$-closed in $R_1 \times R_2$ by Proposition \ref{lemma:free_abelian_p-CSS}. Indeed, following Remark \ref{remark:isolation_hereditary} we see that $\langle g \rangle$ is $p$-isolated in $R_1 \times R_2$. As $\prop(R_1 \times R_2)$ is a restriction of $\prop(G)$, we get that $\langle g \rangle$ is $p$-closed in $G$.
        
        Therefore $G$ is also $p$-CSS.
    \end{proof}
    
    It is natural to ask whether the assumption of direct factors belonging to the class $\Ups$ is necessary. 
    \begin{question}
       Is the class of $p$-CSS groups closed under forming direct products? 
    \end{question}

    The following was proved in \cite{bobr}.
    \begin{theorem}
        \label{theorem:bobrovskii_p}
    Let $G = G_1 \ast_R G_2$ be an amalgam over a common retract, let $g \in G$ be a $p$-isolated element in $G$, and suppose that $G_1$ and $G_2$ are residually $p$-finite. Then $\langle g \rangle$ is not $p$-separable in $G$ if and only if $g$ is conjugate to some $g_i \in G_i$, where $i \in \{1,2\}$, and $\langle g_i \rangle$ is not $p$-separable in $G_i$.
    \end{theorem}
    
    The following three statements provide a pro-$p$ analogue of Lemma \ref{remark:shortening}, Lemma \ref{lemma:gp_separable_elements} and Corollary \ref{corollary:separable_elements}, respectively. We omit the proofs, as they are more-or-less analogous.
    \begin{lemma}
    \label{remark:shortening_p}
    Let $G = \Gamma \mathcal{G}$ be a graph product of residually $p$-finite groups and let $g \in G$ be arbitrary. Then the cyclic subgroup $\langle g \rangle \leq G$ is $p$-separable in $G$ if and only if it is $p$-separable in $G_S$, where $S = \supp(g)$. Furthermore, $\langle g \rangle$ is $p$-separable in $G$ if and only if $\langle g' \rangle$ is $p$-separable in $G$ for some (and hence for all) $g' \in g^G$.
\end{lemma}
\begin{lemma}\label{lemma:gp_separable_p_elements}
    Let $G = \Gamma \mathcal{G}$ be a graph product of residually $p$-finite groups and let $g \in G$ be a cyclically reduced element such that the full subgraph $\Gamma_S$ contains a separating subset, where $S = \supp(g)$. Then the cyclic subgroup $\langle g \rangle \leq G$ is $p$-separable in $G$.
\end{lemma}
\begin{corollary}
    \label{corollary:separable_elements_p}
    Let $\Gamma$ be a graph, $\mathcal{G} = \{G_v \mid v \in V\Gamma\}$ be a family of residually $p$-finite groups and $G = \Gamma \mathcal{G}$ be the corresponding graph product. Suppose that $g \in G$ is an arbitrary element such that $\Gamma_S$ contains a separating subset, where $S = \esupp(g)$. Then the cyclic subgroup $\langle g \rangle \leq G$ is $p$-separable in $G$.
\end{corollary}
    
With all these, we can prove Theorem \ref{proposition:GP_p-CSS}.
\begin{thmB}
For every prime number $p$, the class of $p$-CSS groups in $\Ups$ is closed under forming graph products.
\end{thmB}
\begin{proof}
    The proof of Theorem \ref{proposition:GP_p-CSS} is analogous to the proof of Theorem \ref{proposition:GP_CSS}, modulo the use of Theorem \ref{lemma:primitive_urp_gp}, Lemma \ref{remark:shortening_p}, Lemma \ref{lemma:gp_separable_p_elements}, and Corollary~\ref{corollary:separable_elements_p}.
\end{proof}
The first part of the statement of Corollary \ref{corollary:raags_pcss} follows immediately by applying Theorem \ref{proposition:GP_p-CSS} to graph products of infinite cyclic groups. The second part of the statement follows by Lemma \ref{lemma:characterisation_of_p-isolation}.

\section{$p$-isolated elements of graph products}
\label{section:p-isolated_elements}
The aim of this section is to give a full characterisation of $p$-isolated elements in graph products of groups in $\Ups$.

\subsection{Irreducible factorisations in graph products of groups}
In this subsection we describe a canonical way to factorise elements in graph products into pairwise commuting factors that was introduced in \cite{mf} as the P-S decomposition. 

Let $g \in G=\Gamma\mathcal{G}$ be an element in a graph product. We define $S(g) = \supp(g) \cap \star(\supp(g))$, where the star of a subset of vertices $A\subseteq V$ is defined as
$\star (A)=\cap_{v\in A}\star (v)$.
Similarly, we define $P(g) = \supp(g) \setminus S(g)$. The element $g$ uniquely factorises as a reduced product $g = s(g) p(g)$, where $\supp(s(g)) = S(g)$ and $\supp(p(g)) = P(g)$. We call this factorisation the P-S decomposition of $g$.

Given a graph $\Gamma$, we consider the \emph{complement graph}  $\overline{\Gamma}$ of $\Gamma$, which is defined by $V\overline{\Gamma} = V\Gamma$ and $E\overline{\Gamma} = \binom{V\Gamma}{2} \setminus E\Gamma$. We say that a graph $\Gamma$ is \emph{irreducible} if $\overline{\Gamma}$ is connected, otherwise we say that $\Gamma$ is reducible. Suppose that $\overline{\Gamma}$ can be split into a collection of disjoint connected components $C = \left\{\overline{\Gamma}_1, \overline{\Gamma}_2, \dots \right\}$, then the corresponding collection of full subgraphs $I = \left\{ \Gamma_1, \Gamma_2, \dots \right\}$ of $\Gamma$ is called the \emph{irreducible decomposition} of $\Gamma$ and the members of $I$ are called \emph{irreducible components} of $\Gamma$.
    
Suppose that $\Gamma$ is a graph with at least two vertices and $\mathcal{G} = \{G_v \mid v \in V\Gamma\}$ is a family of non-trivial groups. If the graph is reducible, then the corresponding graph product $G = \Gamma \mathcal{G}$ splits as a direct product $G = \prod_{i \in I}G_i$ where $i$ ranges over the collection of irreducible components of $\Gamma$. Conversely, it was shown by Minasyan and Osin \cite{osin} that if the graph $\Gamma$ is irreducible, then the corresponding graph product is an acylindrically hyperbolic group.
    
Let $g$ be a non-trivial element of the graph product $\Gamma\mathcal{G}$, and consider the full subgraph $\Gamma_S \leq \Gamma$, where $S = \supp(g)$. Let $\left\{\Gamma_1, \dots, \Gamma_d\right\}$ be the irreducible decomposition of $\Gamma_S$, and let $G_1, \dots, G_d \leq G$ be the corresponding full subgroups. Then for every $1 \leq i \leq d$ there is a uniquely given $g_i \in G_i$ such that $g = g_1 \dots g_d$. We refer to this as the \emph{irreducible factorisation} of $g$, and we call the individual elements $g_i$ the \emph{irreducible factors} of $g$. If $d = 1$, i.e. $g$ has only one irreducible factor, then we say that $g$ is \emph{irreducible}. Note that $p(g)$ is the product of all the irreducible factors of lenght at least two and $s(g)$ is exactly the product of all irreducible factors of length one. In particular, if $g$ is irreducible and $|g| > 1$ then $g = p(g)$. This observation allows us to prove the following lemma.

\begin{lemma}\label{lemma:irreducible_element}
    Let $G = \Gamma \mathcal{G}$ be a graph product and let $g \in G$ be irreducible and cyclically reduced. If $|g| > 1$ then $\lvert g^n\rvert = \lvert n\rvert \cdot \lvert g\rvert$ for every $n \in \mathbb{Z}$.
\end{lemma}
\begin{proof}
    The claim is clear if $n=0$. Suppose that $g$ is an irreducible element, and that $n$ is positive. By \cite[Lemma 3.11]{mf} we see that $\FL(g) \cap \LL(g) = \emptyset$ as $p(g) = g$ and $g$ is cyclically reduced with $|g| > 1$. 
    Hence, it is not possible to join together any pair of syllables belonging to consecutive appearances of $g$ in $g^n$.
    Therefore $\lvert g^n\rvert = n \cdot\lvert g\rvert$ for all $n\geq 0$.
    
    For $n$ negative, we notice that $\LL(f^{-1}) = \FL(f)$ and $\lvert f\rvert = \lvert f^{-1}\rvert $ for every $f \in G$.
\end{proof}

\begin{corollary}
    Let $g \in G$ be irreducible and cyclically reduced. If $|g| > 1$ then $\ord(g) = \infty$. 
\end{corollary}

\subsection{$p$-isolated cyclic subgroups}
The aim of this subsection is to characterise the $p$-isolated elements of graph products of $\Ups$-groups. We already proved in Theorem \ref{lemma:primitive_urp_gp} that the class $\Ups$ is closed under forming graph products. Following Lemma \ref{lemma:characterisation_of_p-isolation}, an element is $p$-isolated if and only if its primitive logarithm is a power of $p$, so we only need to describe how to compute primitive logarithms. From now on we assume that $G$ is a graph product of $\Ups$ groups, and in particular that $G\in \Ups$.

\begin{lemma}\label{lemma:irreducible.support}
   Let $g \in G$ be irreducible, cyclically reduced and let $S = \supp(g)$. Then $\plog_G(g) = \plog_{G_S}(g)$. Furthermore, if $|g| > 1$ then $\plog_G(g)$ divides $|g|$.
\end{lemma}
\begin{proof}
The first part of the statement follows from Lemma \ref{lemma:primitive_roots_retract}.

Suppose that $\lvert g\rvert>1$. If $\plog_G(g)=1$ then it divides $\lvert g\rvert$. On the other hand, consider the case when $\plog_G(g)>1$. As $g$ is irreducible, cyclically reduced, and $g=r^{\plog_{G}(g)}$ where $r$ is its root, it must necessarily be that $\plog_G(g)$ divides $\lvert g\rvert$, by Lemma \ref{lemma:irreducible_element}.
\end{proof}

The previous lemma provides an informal algorithm to compute primitive logarithms and primitive roots of cyclically reduced irreducible elements (provided we can solve the word problem in every vertex group): given a reduced word $W_g = (g_1, \dots, g_n)$, where $g_i \in G_{v_i}$ for $i = 1, \dots, n$ and $v_i \in V\Gamma$, by shuffling we can construct the set of all reduced words representing $g$, denote it by $\mathcal{W}_g$, and then (in an increasing order) for every divisor $d$ of $n$ we can check whether there is $W \in \mathcal{W}_g$ with a prefix $W_0$ of length $n/d$ such that $W_0^d =_G W_g$.

\begin{lemma}\label{lemma_last}
   Let $g \in G$ be cyclically reduced and let $g = g_1 \dots g_s$ be its irreducible factorisation. Denote $S_i = \supp(g_i)$ for $i = 1,\dots, s$ and $S = \supp(g)$. Then
   \begin{displaymath}
       \plog_G(g) = \plog_{G_S}(g) = \gcd\left(\plog_{G_{S_1}}(g_1), \dots, \plog_{G_{S_s}}(g_s)\right).
   \end{displaymath}
\end{lemma}
\begin{proof}
    The first equality is proven in Lemma \ref{lemma:irreducible.support}. For the second, notice that $G_S = G_{S_1} \times \dots \times G_{S_s}$. By Theorem \ref{lemma:primitive_urp_gp} and Lemma \ref{lemma:primitive_stability_subgroups} each of the direct factors belongs to $\Ups$, and therefore the rest of the statement follows from Lemma \ref{lemma:plog_direct}.
\end{proof}
Combining Lemma \ref{lemma_last} with Lemma \ref{lemma:characterisation_of_p-isolation}, we obtain:
\begin{corollary}
    Let $g \in G$ be cyclically reduced and let $g = g_1 \dots g_s$ be its irreducible factorisation. Denote $S_i = \supp(g_i)$ for $i = 1,\dots, s$ and $S = \supp(g)$. Then $g$ is $p$-isolated in $G$ if and only if
    \begin{displaymath}
        \gcd\left(\plog_{G_{S_1}}(g_1), \dots, \plog_{G_{S_s}}(g_s)\right) = p^e.
   \end{displaymath}
   for some $e \in \mathbb{N}$.
\end{corollary}

\section*{Acknowledgements}
The authors would like to thank Ashot Minasyan for many useful suggestions. We would also like to thank the anonymous referee for suggesting several simplifications of the original manuscript.

The initial stage of this work greatly benefited from the \textquotedblleft Measured group theory\textquotedblright\ programme held in 2016 at the Erwin Schr\"{o}dinger International Institute for Mathematics and Physics, Vienna. We thank the institution for its warm hospitality and for the excellent working environment. 
At that stage, the authors were supported by the ERC grant of Prof. Goulnara Arzhantseva, grant agreement no.~259527. 

Federico Berlai was supported in part by ERC grant PCG-336983, Basque Government Grant IT974-16, and Ministry of Economy, Industry and Competitiveness of the Spanish Government Grant MTM2017-86802-P, and he is now supported by the Austrian Science Foundation FWF, grant no. J4194.

Michal Ferov is now supported by the Australian Research Council Discovery Project grant DP160100486.

    \bibliographystyle{plain}
    \bibliography{references}

\begin{thebibliography}{10}

\bibitem{allenby}
R.~B. J.~T. Allenby and R.~J. Gregorac.
\newblock On locally extended residually finite groups.
\newblock {\em Conference on Group Theory (Univ. Wisconsin-Parkside, Kenosha,
  Wis., 1972)}, pages 9--17. Lecture Notes in Math., Vol. 319, 1973.

\bibitem{bardakov}
V.~G. Bardakov.
\newblock On a question of {D}.{I}. {M}oldavanski\u{\i}'s question about
  $p$-separable subgroups of a free group.
\newblock {\em Siberian Mathematical Journal}, 45(3):416--419, 2004.

\bibitem{bf}
F.~Berlai and M.~Ferov.
\newblock Residual properties of graph products of groups.
\newblock {\em J. Group Theory}, 19(2):217--231, 2016.

\bibitem{bobr}
P.~A. Bobrovskii and E.~V. Sokolov.
\newblock The cyclic subgroup separability of certain generalized free products
  of two groups.
\newblock {\em Algebra Colloq.}, 17(4):577--582, 2010.

\bibitem{bowditch}
B.~H. Bowditch.
\newblock Relatively hyperbolic groups.
\newblock {\em Internat. J. Algebra Comput.}, 22(3):1250016, 66, 2012.

\bibitem{quasipotency}
J.~Burillo and A.~Martino.
\newblock Quasi-potency and cyclic subgroup separability.
\newblock {\em J. Algebra}, 298(1):188--207, 2006.

\bibitem{hcs_amalgams_zalesskii}
S.~C. Chagas and P.~A. Zalesskii.
\newblock Hereditary conjugacy separability of free products with amalgamation.
\newblock {\em J. Pure Appl. Algebra}, 217(4):598--607, 2013.

\bibitem{evans74}
Benny Evans.
\newblock Cyclic amalgamations of residually finite groups.
\newblock {\em Pacific J. Math.}, 55:371--379, 1974.

\bibitem{mf_thesis}
M.~Ferov.
\newblock {\em Separability properties of graph products of groups}.
\newblock PhD thesis, University of Southampton, 2015.

\bibitem{mf}
M.~Ferov.
\newblock On conjugacy separability of graph products of groups.
\newblock {\em J. Algebra}, 447:135--182, 2016.

\bibitem{green}
E.~R. Green.
\newblock {\em Graph products of groups}.
\newblock PhD thesis, University of Leeds, 1990.

\bibitem{hall}
M.~Hall, Jr.
\newblock Subgroups of finite index in free groups.
\newblock {\em Canadian J. Math.}, 1:187--190, 1949.

\bibitem{MHall50}
M.~Hall, Jr.
\newblock A topology for free groups and related groups.
\newblock {\em Ann. of Math. (2)}, 52:127--139, 1950.

\bibitem{khukhro_nilpotent}
Evgenii~I. Khukhro.
\newblock {\em Nilpotent groups and their automorphisms}, volume~8 of {\em De
  Gruyter Expositions in Mathematics}.
\newblock Walter de Gruyter \& Co., Berlin, 1993.

\bibitem{ls}
R.~C. Lyndon and P.~E. Schupp.
\newblock {\em Combinatorial group theory}.
\newblock Springer-Verlag, Berlin, 1977.
\newblock Ergebnisse der Mathematik und ihrer Grenzgebiete, Band 89.

\bibitem{mks}
W.~Magnus, A.~Karrass, and D.~Solitar.
\newblock {\em Combinatorial group theory: {P}resentations of groups in terms
  of generators and relations}.
\newblock Interscience Publishers [John Wiley \& Sons, Inc.], New
  York-London-Sydney, 1966.

\bibitem{malcev}
A.~I. Mal’cev.
\newblock On homomorphisms onto finite groups.
\newblock {\em American Mathematical Society Translations, Series},
  2(119):67--79, 1983.

\bibitem{mihailova}
K.~A. Miha\u{\i}lova.
\newblock The occurrence problem for free products of groups.
\newblock {\em Mat. Sb. (N.S.)}, 75 (117):199--210, 1968.

\bibitem{ashot_raags}
A.~Minasyan.
\newblock Hereditary conjugacy separability of right-angled {A}rtin groups and
  its applications.
\newblock {\em Groups Geom. Dyn.}, 6(2):335--388, 2012.

\bibitem{osin}
A.~Minasyan and D.~V. Osin.
\newblock Acylindrical hyperbolicity of groups acting on trees.
\newblock {\em Math. Ann.}, 362(3-4):1055--1105, 2015.

\bibitem{mostowski}
A.~Mostowski.
\newblock On the decidability of some problems in special classes of groups.
\newblock {\em Fund. Math.}, 59:123--135, 1966.

\bibitem{olshanskii}
A.~Yu. Ol'shanski\u{\i}.
\newblock On residualing homomorphisms and {$G$}-subgroups of hyperbolic
  groups.
\newblock {\em Internat. J. Algebra Comput.}, 3(4):365--409, 1993.

\bibitem{osin_elementary}
D.~V. Osin.
\newblock Elementary subgroups of relatively hyperbolic groups and bounded
  generation.
\newblock {\em Internat. J. Algebra Comput.}, 16(1):99--118, 2006.

\bibitem{rel_hyp_osin}
D.~V. Osin.
\newblock Relatively hyperbolic groups: intrinsic geometry, algebraic
  properties, and algorithmic problems.
\newblock {\em Mem. Amer. Math. Soc.}, 179(843):vi+100, 2006.

\bibitem{cs_amalgams_zalesskii}
L.~Ribes, D.~Segal, and P.~A. Zalesskii.
\newblock Conjugacy separability and free products of groups with cyclic
  amalgamation.
\newblock {\em J. London Math. Soc. (2)}, 57(3):609--628, 1998.

\bibitem{rz}
L.~Ribes and P.~Zalesskii.
\newblock {\em Profinite groups}, volume~40 of {\em Ergebnisse der Mathematik
  und ihrer Grenzgebiete. 3. Folge. A Series of Modern Surveys in Mathematics
  [Results in Mathematics and Related Areas. 3rd Series. A Series of Modern
  Surveys in Mathematics]}.
\newblock Springer-Verlag, Berlin, second edition, 2010.

\bibitem{stebe}
Peter Stebe.
\newblock Residual finiteness of a class of knot groups.
\newblock {\em Comm. Pure Appl. Math.}, 21:563--583, 1968.

\bibitem{Toinet}
E.~Toinet.
\newblock Conjugacy {$p$}-separability of right-angled {A}rtin groups and
  applications.
\newblock {\em Groups Geom. Dyn.}, 7(3):751--790, 2013.

\end{thebibliography}
\end{document}